\documentclass[12pt]{article}
\usepackage[left=2cm,right=2cm,top=3cm,bottom=2cm]{geometry}
\usepackage{amsmath,amsthm,amssymb}
\usepackage{bm}
\usepackage{graphicx}
\usepackage{mathrsfs}
\usepackage{hyperref}
\usepackage{xcolor}

\newtheorem{theorem}{Theorem}

\newtheorem{remark}{Remark}
\newtheorem{lemma}{Lemma}
\newtheorem{corollary}{Corollary}

\newcommand{\p}{\partial}
\newcommand{\vu}{\bm{u}}

\title{Rigidity for homogeneous solutions to the two-dimensional Euler equations in sector-type domains}

\author{Li Li\footnote{School of Mathematics and Statistics, Ningbo University, 818 Fenghua Road, Ningbo, Zhejiang 315211, China. Email: lili2@nbu.edu.cn.}, Xukai Yan\footnote{Department of Mathematics, Oklahoma State University, 401 Mathematical Sciences Building, Stillwater, OK 74078, USA. Email: xuyan@okstate.edu.}, Zhibo Yang\footnote{Corresponding author. School of Mathematics and Statistics, Ningbo University, 818 Fenghua Road, Ningbo, Zhejiang 315211, China. Email: 2311400053@nbu.edu.cn.}}

\date{\today}

\begin{document}

\maketitle

\begin{abstract}

    We study the rigidity problem for $(-\alpha)$-homogeneous solutions to the two-dimensional incompressible stationary Euler equations in sector-type domains $\Omega_{a, b, \theta_0}:= \{(r,\theta): a<r<b, \ 0<\theta<\theta_0\}$, where $\alpha\in\mathbb{R}$, $0\leqslant a < b \leqslant +\infty$ and $0< \theta_0 \leqslant 2\pi$. For each type of domains, depending on whether $a = 0$ or $a > 0$, and $b = +\infty$ or $b < +\infty$, we show that if a solution satisfies some homogeneity assumptions on the boundary of $\Omega_{a, b, \theta_0}$ and if the radial or angular component of the velocity does not vanish in $\overline{\Omega_{a, b, \theta_0}}\setminus\{\bm{0}\}$, then  it must be homogeneous throughout $\overline{\Omega_{a, b, \theta_0}}\setminus\{\bm{0}\}$.

    \medskip


\end{abstract}

\section{Introduction}
The Euler equations describe the motion of an inviscid flow. We consider the rigidity problem for the two-dimensional incompressible stationary Euler equations
\begin{equation}\label{eq:euler}
    \left\{
    \begin{aligned}
        \vu \cdot \nabla \vu & = - \nabla P & \text{in } \Omega, \\
        \text{div } \vu & = 0 & \text{in } \Omega, 
    \end{aligned}
    \right.
\end{equation}
where $\vu$ is the velocity field, $P$ is the pressure and $\Omega\subseteq \mathbb{R}^2$ is a domain. 

The rigidity problem of (\ref{eq:euler}) concerns whether, for a given domain $\Omega$, the solution $\vu$ must take a specific form under certain conditions.
Hamel and Nadirashvili studied in \cite{HamelNadirashvili2017} the rigidity problem of (\ref{eq:euler}) for two-dimensional strips and half-plane. They proved that all steady flows in a strip with no stagnation point and satisfying tangential boundary conditions are shear flows. The same conclusion holds for bounded steady flows in a half-plane under the similar tangential boundary conditions. 
In a subsequent work, Hamel and Nadirashvili \cite{HamelNadirashvili2019} showed that all bounded steady flows with no stagnation point in the two-dimensional plane must be shear flows. 
In \cite{ConstantinDrivasGinsberg2021}, Constantin, Drivas and Ginsberg established rigidity results for two-dimensional Euler equations (\ref{eq:euler}) and Boussinesq equations in a periodic channel, as well as for three-dimensional axisymmetric Euler equations in a pipe, under certain conditions. Furthermore, in the sense of Arnol'd stable states, solutions to (\ref{eq:euler}) on some two-dimensional compact Riemannian manifolds with smooth boundary (e.g. spherical cap) also exhibit certain rigidity properties. 

Hamel and Nadirashvili proved in \cite{HamelNadirashvili2023} a series of rigidity results for (\ref{eq:euler}) in two-dimensional bounded annuli, exteriors of disks, punctured disks and punctured plane, as well as classification results for (\ref{eq:euler}) in simply or doubly connected bounded domains. Specifically, they proved that if a flow in an annulus is tangent to the boundary and has no stagnation point, then it must be circular. 
Their classification result showed that for two-dimensional simply connected bounded domains, if the flow is tangent to the boundary, has a unique stagnation point in the domain and constant norm on the boundary, then the domain must be a disk and the flow must be circular.
Constantin, Drivas and Ginsberg \cite{ConstantinDrivasGinsberg2022} generalized this classification result to the case of an ideal incompressible magnetohydrodynamic fluid with nested flux surfaces. In two-dimensional simply connected domains, they proved that if the magnetic field and velocity field are never commensurate, then the only possible domain for such equilibria is a disk, and both the velocity and magnetic field must be circular. 

In \cite{KochNadirashviliSereginSverak2009}, Koch, Nadirashvili, Seregin and \v{S}ver\'ak studied bounded ancient solutions to the Navier-Stokes equations in $\mathbb{R}^n\times (-\infty,0)$. They showed that for $n=2$, all bounded weak solutions depend only on time. For $n=3$, they further proved that any bounded axisymmetric weak solution without swirl must be of the form $\bm{u}(\bm{x},t) = (0,0,b(t))$ for some bounded measurable function $b$. 
Later, Lei, Ren and Zhang \cite{LeiRenZhang2022} proved that, in cylindrical coordinates, any bounded ancient mild axisymmetric solution $\bm{v}=v_r \bm{e_r}+v_\theta \bm{e_\theta}+v_z \bm{e_z}$ to the Navier-Stokes equations, with $r v_\theta$ bounded and the solution periodic in the $z$ direction, must be a constant flow of the form $\bm{v} \equiv c\bm{e_z}$ for some constant $c$. 
In \cite{BangGuiWangXie2025}, Bang, Gui, Wang and Xie established the rigidity results for stationary Navier-Stokes equations in a three-dimensional slab. They proved that, under various conditions, the solutions must be trivial, Poiseuille flows, or constant vectors. 
Further studies related to the rigidity problem for the Euler equations and Navier-Stokes equations can be found in the references cited therein.

Equation (\ref{eq:euler}) is invariant under the scaling $(\vu, P)\to (\lambda^\alpha \vu(\lambda \bm{x}),\lambda^{2 \alpha} P(\lambda \bm{x}))$ for any $\alpha\in \mathbb{R}$ and $\lambda>0$. Thus it is natural to study solutions that are invariant under this scaling, which we refer as $(-\alpha)$-homogeneous solutions of (\ref{eq:euler}). Similar homogeneity also holds for Navier-Stokes equations, where $\alpha$ must be $1$ and the corresponding solutions are $(-1)$-homogeneous.

The study of homogeneous solutions of Navier-Stokes equations  can be traced back to \cite{Landau, Slezkin}. Landau discovered a three-parameter family of exact $(-1)$-homogeneous solutions of the Navier-Stokes equations, known as Landau solutions, which are axisymmetric with no swirl. 
Tian and Xin \cite{TianXin} proved that all axisymmetric homogeneous solutions $\bm{u} \in C^2(\mathbb{R}^3 \setminus \{\bm{0} \})$ to the Navier-Stokes equations are Landau solutions. 
Later, \v{S}ver\'{a}k \cite{Sverak} removed the axisymmetry assumption and proved that all homogeneous solutions $\vu \in C^2( \mathbb{R}^3 \setminus \{\bm{0} \})$ are Landau solutions. In the same paper, \v{S}ver\'{a}k also proved that there are no homogeneous solution for the higher-dimensional problem ($n \geqslant 4$), and the homogeneous solution for the two-dimensional problem can be fully classified under a zero flux condition on a circle.

Motivated by modeling of tornadoes, Serrin \cite{Serrin} studied homogeneous solutions of the Navier-Stokes equations in a half-space excluding the positive $x_3$-axis under homogeneous Dirichlet boundary conditions. 
Jiang, Li and Lu \cite{JiangLiLu} recently extended Serrin's results to more general conical domains. They proved that all axisymmetric homogeneous solutions fall into three classes characterized by two parameters depending on the conical angle, and also established corresponding existence and nonexistence results of such solutions. 
In \cite{LiLiYan2018-1}, Li, Li and Yan studied axisymmetric homogeneous solutions of the three-dimensional stationary Navier-Stokes equations that are smooth in $\mathbb{R}^3$ except the negative $x_3$-axis. They classified all such solutions with no swirl, and established the existence and nonexistence of solutions with nonzero swirl nearby the no-swirl solutions. 
Subsequent works \cite{LiLiYan2018-2, LiLiYan2019-3} further established analogous results for axisymmetric homogeneous solutions of stationary Navier-Stokes equations in $\mathbb{R}^3$ smooth away from the entire $x_3$-axis. 
They \cite{LiLiYan2019viscositylimit} also investigated the vanishing viscosity limit of $(-1)$-homogeneous axisymmetric no-swirl solutions to the stationary Navier-Stokes equations in $\mathbb{R}^3\setminus\{x_3\textrm{-axis}\}$. As the viscosity tends to zero, some sequences of solutions of Navier-Stokes equations converge  to solutions of the Euler equations (\ref{eq:euler}) in $C^m_{\rm{loc}}(\mathbb{R}^3\setminus\{x_3\textrm{-axis}\})$ for any integer $m\geqslant 0$, while other sequence exhibit boundary layer or interior layer behaviors. 

Luo and Shvydkov \cite{LuoShvydoky} studied the classification of $(-\alpha)$-homogeneous solutions to the two-dimensional stationary Euler equations with locally finite energy, where $\alpha \leqslant 1$. For $\alpha = 1$, explicit exact solutions can be derived. For all $\alpha$ except those in $(-1/3, 1/4) \setminus \{0 \}$, they obtained a complete classification of such solutions. 
For $1/2<\alpha<1$, they proved only trivial solutions, parallel shear and rotational flows, exist. On the other hand, for $\alpha \leqslant 1/2$, they found numerous solutions, which exhibit hyperbolic, parabolic, or elliptic streamline structures. 
Notably, when $\alpha<-7/2$, the number of distinct non-trivial elliptic solutions is given by $\sharp\{(2, \sqrt{2(1-\alpha)}) \cap \mathbb{N} \}$. 
Shvydkov \cite{Shvydkov} later investigated the classification of $(-\alpha)$-homogeneous solutions to the three-dimensional steady Euler equations and established several nonexistence results, including $(-1)$-homogeneous solutions in $ C^1(\mathbb{R}^3 \setminus \{\bm{0} \})$ and axisymmetric $(-\alpha)$-homogeneous solutions for $0<\alpha<2$. 
Subsequently, Abe \cite{Abe} proved the existence of axisymmetric $(-\alpha)$-homogeneous solutions for all $\alpha \in \mathbb{R} \setminus [0,2]$, along with axisymmetric homogeneous solutions possessing certain special properties. 
For further investigations into homogeneous solutions of the Euler equations and Navier-Stokes equations, see the above citations and the references therein.

In this paper, we study the rigidity problem for homogeneous solutions of the two-dimensional  Euler equations (\ref{eq:euler}) in sector-type domains. Let $(r, \theta)$ be the standard polar coordinates in $\mathbb{R}^2$, and $ \bm{e_\theta}=(-\sin \theta,\cos \theta)^T$, $\bm{e_r}=(\cos \theta,\sin \theta)^T$ be the corresponding unit vectors. A two dimensional vector field $\vu$ can be written as $\vu=u_\theta \bm{e_\theta}+u_r \bm{e_r}$. Let $0\leqslant a<b\leqslant \infty$ and $\theta_0 \in (0,2\pi]$, define the sector-type region
\begin{equation}\label{eq:domain}
    \Omega_{a,b,\theta_0} := \{(r,\theta): a<r<b, \ 0<\theta<\theta_0\}.
\end{equation}
When $a=0$, $\Omega_{a,b,\theta_0}$ represents a standard sector, when $a>0$, it is a truncated sector.

We consider the following question: \emph{given $\alpha\in \mathbb{R}$, under what condition will a solution $(\vu, P)$ of (\ref{eq:euler})  in $\Omega_{a, b, \theta_0}$ be $(-\alpha)$-homogeneous?} In particular, if $(\vu, P)$ is $(-\alpha)$-homogeneous on the boundary of $\Omega_{a, b, \theta_0}$, is it also $(-\alpha)$-homogeneous inside $\Omega_{a, b, \theta_0}$? The conditions vary for different types of domains depending on whether $a=0$ or $a>0$, and $b<\infty$ or $b=\infty$. We will study this problem for the following domains: $\Omega_{1,2,\theta_0}$ (Theorem \ref{thm:1}-\ref{thm:3}), $\Omega_{1,\infty,\theta_0}$ and $\Omega_{0,1,\theta_0}$ (Theorem \ref{thm:4}), and $\Omega_{0,\infty,\theta_0}$ (Theorem \ref{thm:5}). 
The results can be naturally extended to more general domains $\Omega_{a,b,\theta_0}$, $\Omega_{a,\infty,\theta_0}$ and $\Omega_{0,b,\theta_0}$, respectively, for any $0<a<b<\infty$.

If $(\vu,P)$ is a $(-\alpha)$-homogeneous solution of (\ref{eq:euler}), then $\vu$ and $P$ can be expressed as
\begin{equation}\label{eq:u:P:-alpha}
    \vu=\frac{v(\theta)}{r^{\alpha}} \bm{e_\theta}+\frac{f(\theta)}{r^{\alpha}} \bm{e_r}, \qquad P = \frac{p}{r^{2\alpha}},
\end{equation}
for some constant $p$ and functions $v(\theta), f(\theta)$ satisfying
\begin{equation}\label{eq:Euler:-alpha}
    \begin{cases}
    (1-\alpha)f(\theta)+v'(\theta)=0, \\
    v(\theta)f'(\theta)= \alpha f^2(\theta) + v^2(\theta)+2\alpha p.
    \end{cases}
\end{equation}
See Appendix for the derivation of (\ref{eq:u:P:-alpha}) and (\ref{eq:Euler:-alpha}). 
In particular, when $\alpha=1$, the above implies 
\begin{equation}\label{eq:Euler:-1}
    v \equiv c, \qquad v f'(\theta) = f^2(\theta) + v^2 + 2 p,
\end{equation}
for some constant $c$.

Consider the boundary conditions of a $(-\alpha)$-homogeneous solution $(\vu, P) \in C^2(\overline{\Omega_{a,b,\theta_0}}\setminus\{\bm{0}\})$ of (\ref{eq:euler}) on $\partial \Omega_{a, b, \theta_0}\setminus\{\bm{0}\}$. The above implies it satisfies
\begin{equation}\label{eq:BC:hom:-alpha:theta}
    u_\theta|_{\theta=0}=\frac{c_1}{r^{\alpha}}, \quad u_\theta|_{\theta=\theta_0}=\frac{c_2}{r^{\alpha}},
\end{equation}
for some constants $c_1, c_2$. 
Moreover, if $(a, b)=(1, 2)$, then 
\begin{equation}\label{eq:BC:hom:-alpha:r}
     \vu|_{r=1} = 2^{\alpha} \vu|_{r=2}, \quad 
    \p_r u_\theta|_{r=1}=2^{\alpha+1} \p_r u_\theta|_{r=2}. 
\end{equation}
Given a solution of (\ref{eq:euler}) in $\Omega_{a, b, \theta_0}$ satisfying (\ref{eq:BC:hom:-alpha:theta}), and (\ref{eq:BC:hom:-alpha:r}) when $(a, b)=(1, 2)$, we investigate if $(\vu, P)$ is $(-\alpha)$-homogeneous in $\Omega_{a, b, \theta_0}$. Similar as \cite{HamelNadirashvili2017,HamelNadirashvili2019,HamelNadirashvili2023}, we consider the problem for solutions with no stagnation point in $\overline{\Omega_{a,b,\theta_0}}\setminus \{\bm{0} \}$. Here we assume a slightly stronger condition that either $|u_r|>0$ or $|u_{\theta}|>0$ in $\overline{\Omega_{a, b, \theta_0}}\setminus \{\bm{0} \}$. 

\medskip

Our first result is about the rigidity of $(-1)$-homogeneous solutions on $\Omega_{1,2,\theta_0}$.
\begin{theorem}\label{thm:1}
    Let $\theta_0\in (0, 2\pi]$, $\Omega_{1,2,\theta_0}$ be defined by $(\ref{eq:domain})$, $\vu \in C^2(\overline{\Omega_{1,2,\theta_0}})$ be a solution 
    of $(\ref{eq:euler})$ satisfying $|u_r|>0$ in $\overline{\Omega_{1,2,\theta_0}}$. Assume there exist some constants $c_1, c_2$, such that that \eqref{eq:BC:hom:-alpha:theta} and \eqref{eq:BC:hom:-alpha:r} hold for $\alpha=1$. 
    Then $c_1=c_2=c$, and the following hold: 

    {\rm (i)} If $c=0$, then $\vu$ is $(-1)$-homogeneous, $P$ is $(-2)$-homogeneous up to a constant translation, and $(\vu,P)$ is given by 
    \eqref{eq:u:P:-alpha} with $\alpha=1$, $v \equiv 0$, $f \equiv b$ and $p \equiv -\frac{b^2}{2}$ for some constant $b\not=0$. 
    
    {\rm (ii)} If $c \neq 0$, assume $\p_\theta (ru_r)|_{\theta=0}=A$ or $\p_\theta (ru_r)|_{\theta=\theta_0}=A$ for some constant $A$.
    Then $\vu$ is $(-1)$-homogeneous, $P$ is $(-2)$-homogeneous up to a constant translation, and $(\vu,P)$ takes the form in \eqref{eq:u:P:-alpha} with $\alpha = 1$ for some $v,f$ satisfying \eqref{eq:Euler:-1} and $f(\theta)\in C^2([0,\theta_0])$ being strictly sign-definite.
\end{theorem}

To the best of our knowledge, this is the first result concerning the rigidity of solutions to the Euler equations that is characterized by homogeneous solutions. It in particular extends Theorem 1.1 in \cite{ConstantinDrivasGinsberg2021} to the scheme of $(-1)$-homogeneous solutions. In \cite{ConstantinDrivasGinsberg2021}, the rigidity result of the solution $\vu$ to (\ref{eq:euler}) in a periodic channel says that the solution must be a shear flow. 
In Theorem \ref{thm:1}, we propose some boundary conditions that are natural for $(-1)$-homogeneous solutions. Specifically, the conditions imposed on the lines $r=1$ and $r=2$ resemble the periodic boundary conditions in \cite{ConstantinDrivasGinsberg2021}, differing only by a constant scaling factor. The tangential boundary condition in \cite{ConstantinDrivasGinsberg2021} corresponds to $c=0$ in our theorem, while the shear flow solution corresponds to radial flow (i.e. $u_\theta \equiv 0$) in our theorem. 
Moreover, we also consider the case where $c\not=0$ as given in Theorem \ref{thm:1} (ii), for which a rigidity result can be obtained under an extra assumption that $\partial_{\theta}(ru_r)$ is a constant along $\theta=0$ or $\theta=\theta_0$ on $\partial\Omega_{1, 2, \theta_0}$.  

To prove Theorem \ref{thm:1}, we study the stream function $\psi$ of $\vu$. Due to the solenoidal nature of the vector field, any solution $\vu \in C^2(\overline{\Omega_{1,2,\theta_0}}\setminus \{\bm{0} \})$ can be constructed from a stream function $\psi\in C^3(\overline{\Omega_{a,b,\theta_0}}\setminus \{\bm{0} \})$ via $\vu = \nabla^{\bot} \psi$, i.e.
\begin{equation}\label{thm0:eq0}
    u_\theta=\frac{\p \psi}{\p r}, \quad 
    u_r=-\frac{1}{r} \frac{\p \psi}{\p \theta}, \quad (r,\theta)\in \overline{\Omega_{a,b,\theta_0}} \setminus \{ \bm{0} \}.
\end{equation}
The key step is to show that $\psi=c\ln r+h(\theta)$ for some constant $c$ and some function $h(\theta)$ in $\Omega_{a,b,\theta_0}$, so that $\vu = \nabla^{\bot} \psi$ is $(-1)$-homogeneous. 
By the homogeneity assumption of $\vu$ on $\partial\Omega_{a,b,\theta_0}$, $\psi$ is in this form on $\partial\Omega_{a,b,\theta_0}$. Using the assumption that $|u_r|>0$ in $\overline{\Omega_{a,b,\theta_0}}\setminus \{\bm{0} \}$, we show $\psi$ satisfies an elliptic equation 
\begin{equation}\label{eq_psi}
    \Delta \psi=g(\psi)
\end{equation} for some $C^1$ function $g$. 
Then by applying the sliding method, we establish a Liouville-type theorem (see Lemma \ref{lem:4} below) to conclude the rigidity of $\psi$, which implies that $\vu$ must be $(-1)$-homogeneous. 

We extend the above strategy to establish the rigidity of $(-\alpha)$-solutions on $\Omega_{a, b, \theta_0}$ for $\alpha\ne 1$ and $0\leqslant a<b\leqslant \infty$, under similar assumptions to those in Theorem \ref{thm:1} and some extra boundary conditions. Note if $\vu$ is $(-\alpha)$-homogeneous, then $r^{\alpha}u_r$ depends only on $\theta$, and $\partial_{\theta}(r^{\alpha}u_r)= const$ along $\{\theta=0\}$ and $\{\theta=\theta_0\}$. 

\begin{theorem}\label{thm:2}
    Let $\Omega_{1,2,\theta_0}$ be a domain given by $(\ref{eq:domain})$ with $0<\theta_0<2\pi$, $\alpha \neq 1$, $\vu \in C^2(\overline{\Omega_{1,2,\theta_0}})$ is a solution of $(\ref{eq:euler})$ with $|u_r|>0$ in $\overline{\Omega_{1,2,\theta_0}}$. If \eqref{eq:BC:hom:-alpha:theta} and \eqref{eq:BC:hom:-alpha:r} hold for some constants $c_1$ and $c_2$, then $c_1 \neq c_2$. 
    In addition, assume one of the following holds:

    {\rm{(A1)}} $c_1 c_2>0$, and either $\p_\theta (r^{\alpha} u_r)|_{\theta=0}=c_3$ or $\p_\theta (r^{\alpha} u_r)|_{\theta=\theta_0}=c_4$ for some constants $c_3$, $c_4$;

    {\rm{(A2)}} $c_1=0$, $\alpha>1$ and $\p_\theta (r^{\alpha} u_r)|_{\theta=\theta_0}=c_4$ for some constant $c_4$;

    {\rm{(A3)}} $c_2=0$, $\alpha>1$, and $\p_\theta (r^{\alpha} u_r)|_{\theta=0}=c_3$ for some constant $c_3$; 

    {\rm{(A4)}} $c_1 c_2<0$, $\alpha>1$, $\p_\theta (r^{\alpha} u_r)|_{\theta=0}=c_3$ and $\p_\theta (r^{\alpha} u_r)|_{\theta=\theta_0}=c_4$ for some constants $c_3$ and $c_4$.

    Then $\vu$ is $(-\alpha)$-homogeneous, $P$ is $(-2\alpha)$-homogeneous up to a constant translation, and $(\vu,P)$ takes the form in \eqref{eq:u:P:-alpha} for some $v,f\in C^2([0, \theta_0])$ satisfying \eqref{eq:Euler:-alpha}. 
    Moreover, $v(\theta)$ is strictly monotone; 
    $|v(\theta)|>0$ for $\theta \in (0,\theta_0)$ in cases {\rm{(A1)}}-{\rm{(A3)}}, $v(\theta)$ is sign-changing in case {\rm{(A4)}}, and $f(\theta)$ is strictly sign-definite for $\theta \in [0,\theta_0]$ in all four cases.
\end{theorem}

\begin{remark}
    In Theorem \ref{thm:2}, the condition $\alpha>1$ in \textrm{(A2)}-\textrm{(A4)} is imposed to ensure that the function $g$ on RHS of (\ref{eq_psi}) are Lipschitz on the range of $\psi$. 
    In the case when (A1) holds, $0 \notin \bar{J}$ with $J:=\text{dom}(g)$, thus $g$ is always Lipschitz continuous on $\bar{J}$, regardless of the value of $\alpha$.
\end{remark}

\begin{remark}
    In Theorem \ref{thm:2} we consider $0<\theta_0<2\pi$. If $\theta_0=2\pi$ and $\alpha \neq 1$, then any $(-\alpha)$-homogeneous solutions satisfy (\ref{eq:Euler:-alpha}), and the first line of (\ref{eq:Euler:-alpha}) implies that $\int_0^{2\pi} f(\theta) d \theta = 0$. This in turn says that $f$ must vanish at some point in $[0,2\pi]$, which is a contradiction to our condition $|u_r|>0$ in $\overline{\Omega_{1,2,\theta_0}}$.
    Note this does not violate the existence of $(-\alpha)$-homogeneous solutions of the Euler equations (\ref{eq:euler}) in $\mathbb{R}^2 \setminus \{\bm{0} \}$. 
\end{remark}

Theorem \ref{thm:1} and Theorem \ref{thm:2} are established under the condition $|u_r|>0$, which excludes circular flows (i.e. $u_r \equiv 0$). By instead imposing the condition $|u_\theta|>0$ in $\overline{\Omega_{1,2,\theta_0}}$, we establish the following rigidity result when $\alpha\geqslant 1$. 

\begin{theorem}\label{thm:3}
    Let $\alpha \geqslant 1$, $0<\theta_0\leqslant 2\pi$, $\Omega_{1,2,\theta_0}$ be defined by $(\ref{eq:domain})$. Assume $\vu \in C^2(\overline{\Omega_{1,2,\theta_0}})$ is a solution to $(\ref{eq:euler})$ with $|u_\theta|>0$ in $\overline{\Omega_{1,2,\theta_0}}$, satisfying \eqref{eq:BC:hom:-alpha:theta} for some constants $c_1$ and $c_2$, and 
    \begin{equation*}
        u_r|_{\theta=0}= u_r|_{\theta=\theta_0}, \quad 
        u_r|_{r=1}= u_r|_{r=2}=0. 
    \end{equation*}
    Then $c_1 =c_2=c \neq 0$. Moreover, assume 
    \begin{equation*}
        c^{-1}\p_\theta(r^\alpha u_r)|_{\theta=0} \geqslant 1-\alpha \qquad \text{or} \qquad c^{-1}\p_\theta (r^\alpha u_r)|_{\theta=\theta_0} \geqslant 1-\alpha. 
    \end{equation*}
    
    \noindent
    Then $\vu$ is $(-\alpha)$-homogeneous, $P$ is $(-2\alpha)$-homogeneous up to a constant translation, and $(\vu, P)$ is given by \eqref{eq:u:P:-alpha} with $f \equiv 0$, $v \equiv c$, $p \equiv -c^2/2\alpha$. 
\end{theorem}

Theorem \ref{thm:3} generalizes Theorem 1.1 in \cite{HamelNadirashvili2023} to the framework of homogeneous solutions in sector-type domains, where  \cite{HamelNadirashvili2023}  shows  the solution of  (\ref{eq:euler}) is a circular flow if $\theta_0=2\pi$, $|\vu|>0$ and $u_r|_{r=1}= u_r|_{r=2}=0$.

Now we consider regions $\Omega_{a, b, \theta_0}$ when $a=0$ or $b=\infty$. The following result presents a rigidity result for $(-1)$-homogeneous solutions in $\Omega_{0,1,\theta_0}$ and $\Omega_{1,\infty,\theta_0}$, under the assumption that $u_r$ has no stagnation point.

\begin{theorem}\label{thm:4}
    Let $\theta_0 \in (0, 2\pi]$, $(a, b)=(0, 1)$ or $(1, \infty)$, and $\Omega_{a, b,\theta_0}$ be defined by $(\ref{eq:domain})$, $\vu \in C^2(\overline{\Omega_{a, b,\theta_0}} \setminus \{\bm{0} \})$ be a solution to $(\ref{eq:euler})$ satisfying $|u_r|>0$ in $\overline{\Omega_{a, b,\theta_0}}\setminus \{\bm{0} \}$. 
    Assume $u_\theta|_{\p \Omega_{a,b,\theta_0}}=c/r$ for some constant $c$, and one of the following holds in $\overline{\Omega_{a,b,\theta_0}} \setminus \{\bm{0} \}$:
    
    {\rm (B1)} $c=0$, $u_r<0$, and $\p_\theta u_r|_{r=1}\geqslant \p_r(ru_\theta)|_{r=1}$; 
    
    {\rm (B2)} $c=0$, $u_r>0$, and $\p_\theta u_r|_{r=1} \leqslant \p_r(ru_\theta)|_{r=1}$;
    
    {\rm (B3)} $c \vu\cdot \bm{\nu}|_{r=1} > 0$ and $\p_\theta(ru_r)|_{\theta=0}=A$ for some constant $A$;
    
    {\rm (B4)} $c \vu\cdot \bm{\nu}|_{r=1} < 0$ and $\p_\theta(ru_r)|_{\theta=\theta_0}=A$ for some constant $A$, 

    \noindent
    where $\bm{\nu}$ is the outer normal vector of $\partial \Omega_{a,b,\theta_0}$. Then $\vu$ is $(-1)$-homogeneous, $P$ is $(-2)$-homogeneous up to a constant translation, $(\vu, P)$ is in the explicit form \eqref{eq:u:P:-alpha} with $\alpha=1$, $f(\theta)\in C^2([0,\theta_0])$ being strictly sign-definite, and $(v,f)$ satisfying \eqref{eq:Euler:-1}.
\end{theorem}

Finally, we present a rigidity result for $(-\alpha)$-homogeneous solutions in $\Omega_{0,\infty,\theta_0}$.

\begin{theorem}\label{thm:5}
    Let $\alpha \in \mathbb{R}$, $0<\theta_0\leqslant 2\pi$, and $\Omega_{0,\infty,\theta_0}$ be  defined by $(\ref{eq:domain})$, $\vu \in C^2(\overline{\Omega_{0,\infty,\theta_0}} \setminus \{\bm{0} \})$ be a solution to $(\ref{eq:euler})$ satisfying $|u_r|>0$ in $\overline{\Omega_{0,\infty,\theta_0}} \setminus \{\bm{0} \}$. Assume \eqref{eq:BC:hom:-alpha:theta} holds for some constants $c_1$ and $c_2$. 

    {\rm (i)} If $\alpha=1$,   then $c_1=c_2=c$. If $c=0$, assume one of the following holds for some $r_0>0$:

    {\rm (B1')} $c=0$, $u_r<0$, and $\p_\theta u_r|_{r=r_0}\geqslant \p_r(ru_\theta)|_{r=r_0}$;

    {\rm (B2')} $c=0$, $u_r>0$, and $\p_\theta u_r|_{r=r_0} \leqslant \p_r(ru_\theta)|_{r=r_0}$.

    \noindent
    If $c\ne 0$, assume $\p_\theta(ru_r)|_{\theta=0}=A$ or $\p_\theta(ru_r)|_{\theta=\theta_0}=A$ for some constant $A$. Then $\vu$ is $(-1)$-homogeneous, $P$ is $(-2)$-homogeneous up to a constant translation, $(\vu, P)$ is given by \eqref{eq:u:P:-alpha} with $\alpha=1$, $f(\theta)\in C^2([0,\theta_0])$ being strictly sign-definite, and $(v,f)$ satisfying \eqref{eq:Euler:-1}.

    {\rm (ii)} If $\alpha \ne 1$, $0<\theta_0<2\pi$, assume the stream function $\psi$ of $\bm{u}$ satisfies 
    \begin{equation}\label{thm6:eq1500}
        \lim_{r\to \infty} \psi(r, 0)=
        \lim_{r\to \infty} \psi(r, \theta_0), \text{ if } \alpha>1 \quad \text{or} \quad 
        \lim_{r\to 0^+} \psi(r, 0)=
        \lim_{r\to 0^+} \psi(r, \theta_0), \text{ if } \alpha<1. 
    \end{equation}
    Then $c_1\neq c_2$. If one of {\rm (A1)-(A4)} in Theorem \ref{thm:2} holds, then $\vu$ is $(-\alpha)$-homogeneous, $P$ is $(-2\alpha)$-homogeneous up to a constant translation, $(\vu, P)$ is given by \eqref{eq:u:P:-alpha} for some $v,f\in C^2([0, \theta_0])$ satisfying \eqref{eq:Euler:-alpha}. 
    Moreover, $v(\theta)$ is strictly monotone; 
    $|v(\theta)|>0$ for $\theta \in (0,\theta_0)$ in cases {\rm{(A1)}}-{\rm{(A3)}}, $v(\theta)$ is sign-changing in case {\rm{(A4)}}, and $f(\theta)$ is strictly sign-definite for $\theta \in [0,\theta_0]$ in all four cases. 
\end{theorem}

\begin{remark}\label{rem:thm6}
    For $\alpha \neq 1$, Theorem \ref{thm:5} also holds under the following condition instead of (\ref{thm6:eq1500}): 
    $u_r<0$ (resp. $u_r>0$), and there exists $r_0 \in (0,\infty)$ such that 
    \begin{equation}\label{rem6:eq1}
        \int_0^{\theta_0} u_r(r_0,\varphi){\rm{d}} \varphi \geqslant \frac{c_1-c_2}{1-\alpha} r_0^{-\alpha} \
        ({\rm{resp.}}  \int_0^{\theta_0} u_r(r_0,\varphi){\rm{d}} \varphi \leqslant \frac{c_1-c_2}{1-\alpha} r_0^{-\alpha}). 
    \end{equation}
    See the discussion following the proof of Theorem \ref{thm:5} in Section 3.
\end{remark}

\begin{corollary}\label{cor:1}
    Let $0\leqslant a<b\leqslant \infty$, $\theta_0=2\pi$ and $c_1=c_2=c \in \mathbb{R}$. Assume the conditions of Theorem \ref{thm:1} (i), Theorem \ref{thm:4} or Theorem \ref{thm:5} (i) hold for  $\Omega_{a, b, \theta_0}$ with $(a, b)=(1, 2), (0, 1)$ or $(1, \infty)$, or $(0, \infty)$ respectively. Then $\vu$ is $(-1)$-homogeneous, $P$ is $(-2)$-homogeneous up to a constant translation in the corresponding $\Omega_{a, b, \theta_0}$, and $(\vu,P)$ is in the form \eqref{eq:u:P:-alpha} with $\alpha=1$, $v \equiv c$, $f \equiv d$, $p \equiv - (c^2+d^2)/2$ for some constant $d \neq 0$.
\end{corollary}

In \cite{LuoShvydoky}, it was mentioned that all $(-1)$-homogeneous solutions to (\ref{eq:euler}) in the domain $\mathbb{R}^2 \setminus \{ {\bm{0}} \}$ are given by
\begin{equation*}
    \vu=\frac{d_1}{r} \bm{e_\theta}+\frac{d_2}{r} \bm{e_r},
\end{equation*}
for some $d_1, d_2 \in \mathbb{R}$. 
When $\theta_0<2\pi$, a broader class of $(-1)$-homogeneous solutions in $\Omega_{a, b, \theta_0}$ is admitted by Theorem \ref{thm:1}(ii), Theorem \ref{thm:4} and Theorem \ref{thm:5}(a), since the periodicity of the solution with respect to $\theta$ is not required in $\Omega_{a, b, \theta_0}$. 
When $\theta_0=2\pi$, the above result in \cite{LuoShvydoky} is a consequence of our Theorem \ref{thm:3} and Corollary \ref{cor:1}.  

The paper is organized as follows. In Section 2, we present some key lemmas, including some result on an elliptic equation for the stream function and a Liouville-type theorem. Section 3 is devoted to the proofs of Theorems \ref{thm:1}-\ref{thm:5}. The Appendix provides explicit expressions for the $(-\alpha)$-homogeneous solutions of (\ref{eq:euler}), which serve as fundamental examples in domains of the form (\ref{eq:domain}). 

\section{Preliminaries}

Before proving the main theorems, we first establish several useful lemmas and introduce some notations for clarity. 

Let the domain $\Omega_{a,b,\theta_0}$ be defined as in (\ref{eq:domain}). 
We begin by defining the edges of $\overline{\Omega_{a,b,\theta_0}} \setminus \{\bm{0} \}$. 
The top and bottom edges are given by  
\begin{equation*}
    \p \Omega^{T}_{a,b,\theta_0} := \{(r,\theta) \in \p \Omega_{a,b,\theta_0} \setminus \{\bm{0} \}: \theta=\theta_0 \}, \quad
    \p \Omega^{B}_{a,b,\theta_0} := \{(r,\theta) \in \p \Omega_{a,b,\theta_0} \setminus \{\bm{0} \}: \theta=0 \}. 
\end{equation*}
For $a>0$, the left edge is defined as
\begin{equation*}
    \p \Omega^{L}_{a,b,\theta_0} := \{(r,\theta) \in \p \Omega_{a,b,\theta_0}: r=a \}.
\end{equation*}
For $b<\infty$, the right edge is given by 
\begin{equation*}
    \p \Omega^{R}_{a,b,\theta_0} := \{(r,\theta) \in \p \Omega_{a,b,\theta_0}: r=b \}.
\end{equation*}

Next, we define the vertices of $\overline{\Omega_{a,b,\theta_0}} \setminus \{\mathbf{0}\}$.
If $a>0$, the left vertices exist and are 
\begin{equation*}
    \p \Omega^{TL}_{a,b,\theta_0} := (a,\theta_0), \quad 
    \p \Omega^{BL}_{a,b,\theta_0} := (a,0).
\end{equation*}
If $b<\infty$, the right vertices exist and are  
\begin{equation*}
    \p \Omega^{TR}_{a,b,\theta_0} := (b,\theta_0), \quad 
    \p \Omega^{BR}_{a,b,\theta_0} := (b,0).
\end{equation*}

The geometry of the region $\Omega_{a,b,\theta_0}$ and the associated notation for its boundary segments and vertices are shown in Figure \ref{fig:domain} for $0<a<b<\infty$. 

\begin{figure}[htbp]
  \centering
  \includegraphics[width=0.45\textwidth]{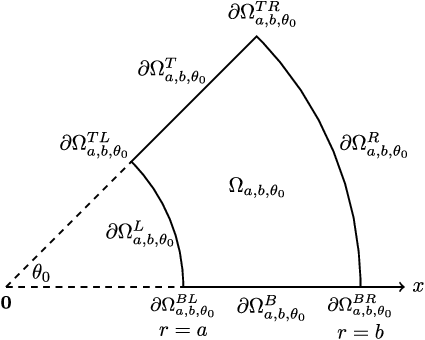}
  \caption{$\Omega_{a,b,\theta_0}$ with its boundary segments and vertices}
  \label{fig:domain}
\end{figure}

The first lemma addresses the local functional dependence between two functions, when their Jacobian determinant vanishes identically. 

\begin{lemma}\label{lem:1}
    Let $\Omega\subseteq \mathbb{R}^2$ be an open domain, $f_1, f_2 \in C^1(\Omega)$. Suppose that 
    $|\nabla f_2|>0$ in $\Omega$, and the Jacibian determinant satisfies 
    \begin{equation*}
        \bigg| \frac{\p (f_1,f_2)}{\p (x,y)} \bigg| \equiv 0, \quad  (x,y) \in \Omega.
    \end{equation*}
    Then for any $(x_0,y_0)\in \Omega$, there exists an open set $U \subset \Omega$ containing $(x_0,y_0)$ and a $C^1$ function $g: f_2(U) \to \mathbb{R}$, such that 
    \begin{equation*}
        f_1(x,y)=g(f_2(x,y)), \quad (x,y) \in U.  
    \end{equation*}
\end{lemma}

\begin{proof}
    For any fixed point $(x_0,y_0) \in \Omega$, there exists an open ball $V \subseteq \Omega$ containing $(x_0,y_0)$. 
    Note $V$ is simply-connected and $f_2 \in C^1(\Omega)$ satisfying $|\nabla f_2|>0$, the image $J:=f_2(V)$ is an open interval. Let
    \begin{equation*}
        F(x,y,c):=f_2(x,y)-c,  \quad (x,y,c)\in V \times J, 
    \end{equation*}
    and set $c_0=f_2(x_0,y_0)$. 
    Note $|\nabla f_2|>0$ in $\Omega$. Without loss of generality, we assume $\p_y f_2(x_0,y_0) \neq 0$. Then we have $F(x_0,y_0,c_0)=0$ and $\p_y F(x_0,y_0,c_0)=\p_y f_2(x_0,y_0) \neq 0$. By the implicit function theorem, there exist open intervals $I_{x_0}$, $I_{y_0}$, $I_{c_0}$ and a $C^1$ function $h: I_{x_0} \times I_{c_0} \to I_{y_0}$, such that 
    \begin{equation*}
        (x_0,y_0,c_0) \in I:=I_{x_0} \times I_{y_0} \times I_{c_0} \subseteq V \times J, 
    \end{equation*}
    \begin{equation*}
        F(x,h(x,c),c)=0, \quad (x,c) \in I_{x_0} \times I_{c_0}, \quad h(x_0, c_0)=y_0, 
    \end{equation*}
    and $\p_y F(x,y,c)= \partial_yf_2(x, y)\neq 0$ in $I$. 

    Define
    \begin{equation*}
        g(x,c):=f_1(x,h(x,c)), \quad (x,c) \in I_{x_0} \times I_{c_0}.
    \end{equation*}
    Then we have $g\in C^1(I_{x_0} \times I_{c_0})$ and
    \begin{equation*}
        \p_x g(x,c) =\p_x f_1 +\p_y f_1\cdot \p_x h 
        =\p_x f_1 +\p_y f_1\cdot \big(-\frac{\p_x f_2}{\p_y f_2} \big) = \frac{1}{\p_y f_2} \cdot \bigg| \frac{\p (f_1,f_2)}{\p (x,y)} \bigg| = 0.
    \end{equation*}
    So $g$ depends only on $c$ in $I_{x_0} \times I_{c_0}$, and then for any $(x,c)\in I_{x_0} \times I_{c_0}$, 
    \begin{equation*}
        f_1(x,h(x,c))=g(c)=g(f_2(x,h(x,c))). 
    \end{equation*}
    Note we have $\partial_yf_2(x, y)\neq 0$ in $I$, and then $\p_ch(x,c)=-\frac{\p_c F}{\p_y F}=\frac{1}{\p_y f_2}$ is strictly sign-definite in $I_{x_0}\times I_{c_0}$. So for each fixed $x\in I_{x_0}$,  $h$ is continuous and monotone in $c\in I_{c_0}$. Thus $h(x, I_{c_0})$ must be an interval, and $y_0$ is an interior point of $h(x_0, I_{c_0})$. Let $U=\cup_{x\in I_{x_0}}\{(x, y)\mid y=h(x, c), c\in I_{c_0}\}$. By the continuity of $h$, we have $U\subseteq I_{x_0} \times I_{y_0} \subseteq \Omega$ is an open set containing $(x_0,y_0)$. In particular, the above implies 
    \begin{equation*}
        f_1(x,y)=g(f_2(x,y)), \quad (x,y)\in U. 
    \end{equation*}
    This completes the proof of the lemma. 
\end{proof}

The second lemma concerns the geometric properties of the level sets of the stream function within $\Omega_{a,b,\theta_0}$. 

\begin{lemma}\label{lem:2}
    Let $0\leqslant a<b\leqslant \infty$, $\theta_0\in (0, 2\pi]$, $\Omega_{a,b,\theta_0}$ be defined by $(\ref{eq:domain})$, and  $\psi(r,\theta) \in C^1(\overline{\Omega_{a,b,\theta_0}} \setminus \{ \bm{0} \})$. Assume one of the following holds: 

    {\rm{(i)}} $|\p_\theta \psi|>0$ in $\overline{\Omega_{a,b,\theta_0}} \setminus \{\bm{0} \}$, and each of $\psi(r,0)$ and $\psi(r,\theta_0)$ is either a constant or strictly monotone for $r \in (a,b)$. 

    {\rm{(ii)}} $0<a<b<\infty$, $|\p_r \psi|>0$ in $\overline{\Omega_{a,b,\theta_0}}$, and each  of $\psi(a,\theta)$ and $\psi(b,\theta)$ is either a constant or strictly monotone for $\theta \in (0,\theta_0)$.
 
    Then for any $c\in \psi(\Omega_{a,b,\theta_0})$, the level set $L_c:=\{(r,\theta)\in \Omega_{a,b,\theta_0}, \psi(r,\theta)=c \}$ must be a $C^1$ simple curve. Moreover, if $|\p_\theta \psi|>0$ in $\Omega_{a,b,\theta_0}$, then $L_c$ can be represented as $\theta=\theta(r),\ r\in I_{c}$ for some open  interval $I_{c} \subseteq (a,b)$; 
    if $|\p_r \psi|>0$ in $\Omega_{a,b,\theta_0}$, then  $L_c$ can be represented as $r=r(\theta),\ \theta\in I_{c}'$ for some open interval $I_{c}' \subseteq (0,\theta_0)$. 
\end{lemma}

\begin{proof}

    (i) Under condition (i), we may assume without loss of generality that $\p_\theta \psi>0$ in $\overline{\Omega_{a,b,\theta_0}} \setminus \{ \bm{0} \}$. Denote $J:=\psi(\Omega_{a,b,\theta_0})$.
    
    \medskip
    
    \noindent\textbf{Case 1.} $\psi(r,0)=c_1$ and $\psi(r,\theta_0)=c_2$ are both constants. 
    
    \medskip
    
    Obviously, $c_1<c_2$ and $J=(c_1,c_2)$ in this case. For any fixed $r\in (a,b)$, $\psi(r,\theta)$ is strictly increasing in $\theta \in [0,\theta_0]$. Consequently, for any $c\in J$ and fixed $r \in (a,b)$, there exists a unique $\theta_r \in (0,\theta_0)$ such that $\psi(r,\theta_r)=c$, and hence $(r,\theta_r) \in L_c$. Since $\p_\theta \psi>0$ in $\Omega_{a,b,\theta_0}$, the implicit function theorem ensures that $L_c$ is a $C^1$ curve, which can be parameterized as $\theta=\theta(r)$ for $r \in (a,b)$. 

    \medskip
    \noindent  \textbf{Case 2.} $\psi(r,0)$ and $\psi(r,\theta_0)$ are strictly monotonic with respect to $r \in (a,b)$. 

    \medskip
    Regarding the domain configurations, we examine the following four cases: $0<a<b<\infty$, $0=a<b<\infty$, $0<a<b=\infty$, and $0=a<b=\infty$. 
    Concerning the monotonicity properties of $\psi$, we restrict our analysis to two distinct scenarios: 
    $\psi(r,0)$ and $\psi(r,\theta_0)$ are strictly increasing in $r \in (a,b)$, or $\psi(r,0)$ is strictly increasing while $\psi(r,\theta_0)$ is strictly decreasing. Other scenarios can be treated similarly.

    \medskip

    \textbf{(2a)} $0<a<b<\infty$, $\psi(r,0)$ and $\psi(r,\theta_0)$ are strictly increasing in $r \in (a,b)$. 
    
    \medskip
    
    Since $\p_\theta \psi>0$ in $\overline{\Omega_{a,b,\theta_0}}$, it follows that $\psi(a,\theta)$ and $\psi(b,\theta)$ are strictly increasing in $\theta \in (0,\theta_0)$. Therefore, the minimum of $\psi$ in $\overline{\Omega_{a,b,\theta_0}}$ can only be attained on $\p \Omega^{BL}_{a,b,\theta_0}$, the maximum can only be attained on $\p \Omega^{TR}_{a,b,\theta_0}$, and $J \subset \mathbb{R}$ is a bounded open interval. 
   
    By the monotonicity of $\psi$ on the boundary, we know that for any $c \in J$, there exists a unique point $(r_1,\theta_1) \in \p \Omega^T_{a,b,\theta_0} \cup \p \Omega^L_{a,b,\theta_0}$, and a unique point $(r_2,\theta_2) \in \p \Omega^B_{a,b,\theta_0} \cup \p \Omega^R_{a,b,\theta_0}$ with $r_1<b$, $\theta_1>0$, $r_2>a$, $\theta_2<\theta_0$, such that $\psi(r_1,\theta_1)=\psi(r_2,\theta_2)=c$. 
    Note we must have $r_1<r_2$. Otherwise if $r_2 \leqslant r_1<b$, then $\theta_2=0$ and
    \begin{equation*}
        c=\psi(r_2,\theta_2) \leqslant \psi(r_1,\theta_2)<\psi(r_1,\theta_1)=c,
    \end{equation*}
    since $\psi(r,0)$ is strictly increasing in $r$ and  $\p_\theta \psi>0$. This leads to a contradiction. 
    
    We now claim that for any $(r,\theta) \in L_c$, the inequality $r_1<r<r_2$ holds. 
    First, consider the case $r_2=b$. By definition of $L_c$, it follows immediately that $r < r_2$. 
    Next, suppose $r_2<b$. Then $\theta_2=0$. For any $(r_*,\theta_*) \in \Omega_{a,b,\theta_0}$ with $r_* \geqslant r_2$, the monotonicity of $\psi$ on the boundary and in the interior yields 
    \begin{equation*}
        \psi(r_*,\theta_*)>\psi(r_*,0) \geqslant \psi(r_2,\theta_2)=c, 
    \end{equation*}
    which implies $(r_*,\theta_*) \notin L_c$. Consequently, for any $(r,\theta) \in L_c$, we must have $r<r_2$. A 
    similar argument shows that $r>r_1$ for any $(r,\theta) \in L_c$, thus establishing the claim.  
    
    Now, fix an arbitrary $r_0 \in (r_1,r_2)$. The boundary monotonicity of $\psi$ gives 
    \begin{equation*}
        \psi(r_0,0)<\psi(r_2,\theta_2)=c=\psi(r_1,\theta_1)<\psi(r_0,\theta_0).
    \end{equation*}
    Moreover, since $\psi$ is strictly increasing with respect to $\theta$ in $\Omega_{a,b,\theta_0}$,  
    there exists a unique $\theta_{r_0} \in (0,\theta_0)$ satisfying $\psi(r_0, \theta_{r_0})=c$, which implies that $(r_0, \theta_{r_0}) \in L_c$. Moreover, since $\p_\theta \psi>0$ in $\Omega_{a,b,\theta_0}$, by the implicit function theorem, we conclude that $L_c:=\{(r,\theta)\in \Omega_{a,b,\theta_0}, \psi(r,\theta)=c \}$ must be a $C^1$ simple curve, which can be parameterized as $\theta=\theta(r)$ for $r\in (r_1,r_2)$. 

    \medskip

    \textbf{(2b)}  $0<a<b<\infty$, $\psi(r,0)$ is strictly increasing and $\psi(r,\theta_0)$ is strictly decreasing in $r \in (a,b)$. 

    \medskip

    Since $\p_\theta \psi>0$ in $\overline{\Omega_{a,b,\theta_0}}$, it follows that $\psi(a,\theta)$ and $\psi(b,\theta)$ are strictly increasing in $\theta \in (0,\theta_0)$. Then  the minimum of $\psi$ in $\overline{\Omega_{a,b,\theta_0}}$ is attained on $\p \Omega^{BL}_{a,b,\theta_0}$, the maximum is attained on $\p \Omega^{TL}_{a,b,\theta_0}$, and $J \subset \mathbb{R}$ is a bounded open interval. 
    By the monotonicity of $\psi$, it follows that for any $c \in J$, there exist exactly two distinct points 
    \begin{equation*}
        (a,\theta_1) \in \p \Omega^L_{a,b,\theta_0}, \quad (r_2,\theta_2) \in \p \Omega^B_{a,b,\theta_0} \cup \p \Omega^R_{a,b,\theta_0} \cup \p \Omega^T_{a,b,\theta_0}
    \end{equation*}
    with $r_2>a$, such that $\psi(a,\theta_1)=\psi(r_2,\theta_2)=c$. 

    By similar arguments as in the proof of (2a), we have $a<r<r_2$ for any $(r,\theta) \in L_c$. 
    Moreover, for any $r_0 \in (a, r_2)$, there exists a unique $\theta_{r_0} \in (0,\theta_0)$ such that $\psi(r_0,\theta_{r_0})=c$. Then the implicit function theorem implies that $L_c$ must be a $C^1$ simple curve, which can be parameterized as $\theta=\theta(r)$, $r\in (a,r_2)$. 

    \medskip

    \textbf{(2c)} $0=a<b<\infty$, $\psi(r,0)$ and $\psi(r,\theta_0)$ are strictly increasing in $r \in (0,b)$.

    \medskip

    Since $\p_\theta \psi>0$ in $\overline{\Omega_{0,b,\theta_0}} \setminus \{\bm{0} \}$, it follows that $\psi(b,\theta)$ is strictly increasing in $\theta \in (0,\theta_0)$. For any fixed $r \in (0,b)$, we have $\psi(r,0)<\psi(r,\theta_0)$, which implies 
    \begin{equation*}
        \lim_{r \to 0^+} \psi(r,0) \leqslant \lim_{r \to 0^+} \psi(r,\theta_0). 
    \end{equation*}
    Therefore, the maximum of $\psi$ in $\overline{\Omega_{0,b,\theta_0}} \setminus\{ \bm{0} \}$ can only be attained on $\p \Omega^{TR}_{0,b,\theta_0}$, while the minimum cannot be achieved in $\overline{\Omega_{0,b,\theta_0}} \setminus\{ \bm{0} \}$. The range $J=(\displaystyle \lim_{r \to 0+} \psi(r,0),\psi(\p \Omega^{TR}_{0,b,\theta_0}))$ forms an open interval. 
    We decompose $J$ into two disjoint subsets:  
    \begin{equation*}
        J=J_1 \cup J_2,
    \end{equation*}
    where $J_1:=\psi(\p \Omega^{T}_{0,b,\theta_0} \setminus \p \Omega^{TR}_{0,b,\theta_0})$ is nonempty, and $J_2:= J \setminus J_1$. Note that $J_2$ is empty if $ \lim_{r \to 0^+} \psi(r,0)=\lim_{r \to 0^+} \psi(r,\theta_0)$. 

    For any $c \in J$, there exists a unique point $(r_2,\theta_2) \in \p \Omega^B_{0,b,\theta_0} \cup \p \Omega^R_{0,b,\theta_0}$ with $r_2>0$ and $\theta_2<\theta_0$ satisfying $\psi(r_2,\theta_2)=c$. 
    
    If $c \in J_1$, there exists a unique point $(r_1,\theta_0) \in \p \Omega^T_{0,b,\theta_0}$ with $0<r_1<b$, such that $\psi(r_1,\theta_0)=\psi(r_2,\theta_2)=c$. Following the same approach as in Case (2a) and (2b), we can show that $r_1<r_2$ and that $L_c$ forms a $C^1$ simple curve  parameterized by $\theta=\theta(r)$ for $r\in (r_1,r_2)$. 

    If $J_2 \neq \emptyset$ and $c \in J_2$, the monotonicity of $\psi$ guarantees that for any $(r,\theta) \in L_c$, we have $0<r<r_2$. Moreover, for each $r_0 \in (0,r_2)$, there exists a unique $\theta_{r_0} \in (0,\theta_0)$ such that $(r_0,\theta_{r_0}) \in L_c$. Consequently, by the implicit function theorem, $L_c$ is a $C^1$ curve that admits the parameterization $\theta=\theta(r)$ for $r \in (0, r_2)$.   

    \medskip

    \textbf{(2d)} $0=a<b<\infty$, $\psi(r,0)$ is strictly increasing and $\psi(r,\theta_0)$ is strictly decreasing in $r \in (0,b)$. 

    \medskip

    In this case, the extreme values of $\psi$ cannot be achieved in $\overline{\Omega_{0,b,\theta_0}} \setminus\{ \bm{0} \}$, $J=(\lim_{r \to 0+} \psi(r,0),$ $\lim_{r \to 0+} \psi(r,\theta_0))$ forms an open interval. 
    For each $c \in J$, there exists a unique point $(r_1,\theta_1) \in \p \Omega_{0,b,\theta_0} \setminus\{ \bm{0} \}$ satisfying $\psi(r_1,\theta_1)=c$. The monotonicity of $\psi$ implies that for any $(r,\theta) \in L_c$, we have $0<r<r_1$. For every $r_0 \in (0,r_1)$, there exists a unique $\theta_{r_0} \in (0,\theta_0)$ such that $(r_0,\theta_{r_0}) \in L_c$. Consequently, by the implicit function theorem, $L_c$ is a $C^1$ curve admitting the representation $\theta=\theta(r)$ for $r \in (0,r_1)$.   
    
    \medskip
    
    \textbf{(2e)} $0<a<b=\infty$, $\psi(r,0)$ and $\psi(r,\theta_0)$ are strictly increasing in $r \in (a,\infty)$.

    \medskip

    In this case, 
    \begin{equation*}
        \lim_{r \to +\infty} \psi(r,0) \leqslant \lim_{r \to +\infty} \psi(r,\theta_0). 
    \end{equation*} 
    The minimum of $\psi$ in $\overline{\Omega_{a,\infty,\theta_0}}$ can only be reached in $\p \Omega^{BL}_{a,\infty,\theta_0}$, while the maximum cannot be achieved in $\overline{\Omega_{a,\infty,\theta_0}}$. 
    Therefore, $J=(\psi(\p \Omega^{BL}_{a,\infty,\theta_0}), \displaystyle \lim_{r \to +\infty} \psi(r,\theta_0))$. 
     We decompose $J$ into two disjoint subsets:  
    \begin{equation*}
        J=J_1 \cup J_2,
    \end{equation*}
    where $J_1:=\psi(\p \Omega^{B}_{a,\infty,\theta_0} \setminus \p \Omega^{BL}_{a,\infty,\theta_0})$ is nonempty, $J_2:= J \setminus J_1$.

    For any $c \in J$, there exists a unique point $(r_1,\theta_1) \in \p \Omega^T_{a,\infty,\theta_0} \cup \p \Omega^L_{a,\infty,\theta_0}$ with $\theta_1>0$ satisfying $\psi(r_1,\theta_1)=c$. 
    
    If $c \in J_1$, there exists a unique point $(r_2,0) \in \p \Omega^B_{a,\infty,\theta_0}$ with $r_2>a$, such that $\psi(r_2,0)=c$. As in the proof of Case (2c), we can show that $r_1<r_2$ and that $L_c$ forms a $C^1$ simple curve parameterized by $\theta=\theta(r)$ for $ r\in (r_1,r_2)$. 

    If $J_2 \neq \emptyset$ and $c \in J_2$, then as in the proof of Case (2c), $L_c$ is a $C^1$ curve that can be represented as $\theta=\theta(r)$ for $r \in (r_1,+\infty)$.    

    \medskip

    \textbf{(2f)} $0<a<b=\infty$, $\psi(r,0)$ is strictly increasing and $\psi(r,\theta_0)$ is strictly decreasing in $r \in (a,\infty)$. 

    \medskip

    In this case, we have
    \begin{equation*}
         \lim_{r\to +\infty} \psi(r,0) \leqslant \lim_{r \to +\infty} \psi(r,\theta_0), 
    \end{equation*}
    and 
    \begin{equation*}
        \inf_{\Omega_{a,\infty,\theta_0}} \psi=\psi(\Omega^{BL}_{a,\infty,\theta_0}), \quad
        \sup_{\Omega_{a,\infty,\theta_0}} \psi=\psi(\Omega^{TL}_{a,\infty,\theta_0}). 
    \end{equation*}
    So $J=(\psi(\Omega^{BL}_{a,\infty,\theta_0}),\psi(\Omega^{TL}_{a,\infty,\theta_0}))$. 
    We decompose $J$ into three disjoint subsets: 
    \begin{equation*}
        J=J_1 \cup J_2 \cup J_3,
    \end{equation*}
    where $J_1:=\psi(\p\Omega_{a,\infty,\theta_0}^B \setminus \p\Omega_{a,\infty,\theta_0}^{BL})$, $J_2:=\psi(\p\Omega_{a,\infty,\theta_0}^T \setminus \p\Omega_{a,\infty,\theta_0}^{TL})$, $J_3:=J\setminus(J_1 \cup J_2)$ are all nonempty sets. 

    For any $c \in J$, there exists a unique point $(a,\theta_1) \in \p \Omega^L_{a,\infty,\theta_0}$ satisfying $\psi(a,\theta_1)=c$ and $0<\theta_1<\theta_0$.  
    
    If $c \in J_1$, there exists a unique point $(r_2,0) \in \p \Omega_{a,\infty,\theta_0}^B$ satisfying $\psi(r_2,0)=c$ and $r_2>a$. As in the previous discussion, we conclude that $L_c$ is a $C^1$ simple curve parameterized as $\theta=\theta(r)$ for $r\in (a,r_2)$. 

    If $c \in J_2$, there is a unique point $(r_2',\theta_0)\in \p \Omega_{a,\infty,\theta_0}^T$ satisfying $\psi(r_2',\theta_0)=c$ and $r_2'>a$. Analogously, $L_c$ is a $C^1$ simple curve, which can be represented as $\theta=\theta(r)$ for $r\in (a,r_2')$. 

    If $c \in J_3$, by the definition of $J_3$, for any fixed $r_0 \in (a,\infty)$, we have
    \begin{equation*}
        \psi(r_0,0)<\lim_{r\to +\infty} \psi(r,0) \leqslant c \leqslant \lim_{r \to +\infty} \psi(r,\theta_0)<\psi(r_0,\theta_0).
    \end{equation*}
    Therefore, for any fixed $r_0 \in (a,\infty)$, there exists a unique $\theta_{r_0} \in (0,\theta_0)$ such that $(r_0,\theta_{r_0}) \in L_c$. Consequently, $L_c$ forms a $C^1$ simple curve parameterized as $\theta=\theta(r)$ for $r\in (a,\infty)$. 

    \medskip

    \textbf{(2g)} $0=a<b=\infty$, $\psi(r,0)$ and $\psi(r,\theta_0)$ are strictly increasing in $r \in (0,\infty)$.

    \medskip

    In this case, 
    \begin{equation*}
        \lim_{r\to 0^+} \psi(r,0) \leqslant \lim_{r \to 0^+} \psi(r,\theta_0), \quad 
        \lim_{r\to +\infty} \psi(r,0) \leqslant \lim_{r \to +\infty} \psi(r,\theta_0),
    \end{equation*}
    and $\psi$ attains neither a maximum nor a minimum in $\overline{\Omega_{0,\infty,\theta_0}}\setminus \{ \bm{0} \}$. 
    Therefore, the interval $J=(\lim_{r\to 0^+} \psi(r,0),$ $\lim_{r \to +\infty} \psi(r,\theta_0))$. 
    
    We decompose $J$ into four disjoint subsets:  
    \begin{equation*}
        J=J_1 \cup J_2 \cup J_3 \cup J_4,
    \end{equation*}
    where $J_1:=\psi(\p\Omega_{0,\infty,\theta_0}^B) \cap \psi(\p\Omega_{0,\infty,\theta_0}^T)$, $J_2:=\psi(\p\Omega_{0,\infty,\theta_0}^B) \setminus J_1$, 
    $J_3:=\psi(\p\Omega_{0,\infty,\theta_0}^T) \setminus J_1$ and 
    $J_4:=J\setminus (J_1 \cup J_2 \cup J_3)$. 
    
    For any $c \in J$, there are four possible cases: 

    If $J_1 \neq \emptyset$ and $c \in J_1$, there exist exactly two distinct points $(r_1,\theta_0) \in \p \Omega_{0,\infty,\theta_0}^T$ and $(r_2,0) \in \p \Omega_{0,\infty,\theta_0}^B$ such that $\psi(r_1,\theta_0)=\psi(r_2,0)=c$. Following a similar argument as before, we deduce that $r_1<r_2$ and $L_c$ is a $C^1$ simple curve parameterized as $\theta=\theta(r)$ for $r\in (r_1,r_2)$. 

    If $J_2 \neq \emptyset$ and $c \in J_2$, there exists a unique point $(r_2',0) \in \p \Omega_{0,\infty,\theta_0}^B$ such that $\psi(r_2',0)=c$. As in the proof of Case (2c), $L_c$ is a $C^1$ simple curve, which can be represented as $\theta=\theta(r)$ for $r\in (0,r_2')$.

    If $J_3 \neq \emptyset$ and $c \in J_3$, we can find a unique point $(r_1',\theta_0) \in \p \Omega_{0,\infty,\theta_0}^T$ such that $\psi(r_1',\theta_0)=c$. As in the proof of Case (2e), $L_c$ is a $C^1$ simple curve parameterized as $\theta=\theta(r)$ for $r\in (r_1',+\infty)$.
    
    If $J_4 \neq \emptyset$ and $c \in J_4$, then as in the proof of Case (2f), we conclude that $L_c$ is a $C^1$ simple curve parameterized as $\theta=\theta(r)$ for $r\in (0,+\infty)$. 

    \medskip

    \textbf{(2h)} $0=a<b=\infty$, $\psi(r,0)$ is strictly increasing and $\psi(r,\theta_0)$ is strictly decreasing in $r \in (a,\infty)$. 

    \medskip

    In this case, 
    \begin{equation*}
        \sup_{r \in (0,\infty)} \psi(r,0)=\lim_{r \to \infty} \psi(r,0) \leqslant \lim_{r \to \infty} \psi(r,\theta_0)= \inf_{r\in (0,\infty)} \psi(r,\theta_0),
    \end{equation*}
    and $\psi$ attains neither a maximum nor a minimum in $\overline{\Omega_{0,\infty,\theta_0}}\setminus \{ \bm{0} \}$. 
    Therefore, the interval $J=(\lim_{r\to 0^+} \psi(r,0),$ $\lim_{r \to 0^+} \psi(r,\theta_0))$. 
    
    We decompose $J$ into three disjoint subsets:  
    \begin{equation*}
        J=J_1 \cup J_2 \cup J_3,
    \end{equation*}
    where $J_1:=\psi(\p\Omega_{0,\infty,\theta_0}^B)$, $J_2:=\psi(\p\Omega_{0,\infty,\theta_0}^T)$, $J_3:=J\setminus(J_1 \cup J_2)$ are all nonempty sets. 

    For any $c \in J$, there are three possible cases:
    
    If $c \in J_1$, there exists a unique point $(r_2,0) \in \p \Omega_{0,\infty,\theta_0}^B$ satisfying $\psi(r_2,0)=c$. Following the argument in Case (2c), $L_c$ is a $C^1$ simple curve parameterized as $\theta=\theta(r)$ for $r\in (0,r_2)$. 

    If $c \in J_2$, there is a unique point $(r_1,\theta_0) \in \p \Omega_{0,\infty,\theta_0}^T$ such that $\psi(r_1,\theta_0)=c$. Analogous to Case (2c), $L_c$ is a $C^1$ simple curve, which can be represented as $\theta=\theta(r)$ for $r\in (0,r_1)$. 

    If $c \in J_3$, the same argument as in Case (2g) shows that $L_c$ is a $C^1$ simple curve parameterized as $\theta=\theta(r)$ for $r\in (0,+\infty)$. 
  
    \medskip

    \noindent \textbf{Case 3.} One of $\psi(r,0)$ and $\psi(r,\theta_0)$ is constant, and the other is strictly monotone in $r \in (a,b)$.
   
    \medskip

    We restrict our analysis to the scenario where $\psi(r,0)$ is constant and $\psi(r,\theta_0)$ is strictly increasing in $r \in (a,b)$, within the following four domains: $0<a<b<\infty$, $0=a<b<\infty$, $0<a<b=\infty$, and $0=a<b=\infty$. Other scenarios can be treated similarly. For convenience, we denote $C:=\psi(r,0)$.

    \medskip

    \textbf{(3a)} $0<a<b<\infty$. 
    
    \medskip
    
    Since $\p_\theta \psi>0$ in $\overline{\Omega_{a,b,\theta_0}}$, it follows that $\psi(a,\theta)$ and $\psi(b,\theta)$ are strictly increasing in $\theta \in (0,\theta_0)$. Therefore, the minimum of $\psi$ in $\overline{\Omega_{a,b,\theta_0}}$ can only be attained on $\p \Omega^{B}_{a,b,\theta_0}$, the maximum can only be attained on $\p \Omega^{TR}_{a,b,\theta_0}$, and $J$ is a bounded open interval. 
   
    By the monotonicity of $\psi$ on the boundary, we know that for any $c \in J$, there exists a unique point $(r_1,\theta_1) \in \p \Omega^T_{a,b,\theta_0} \cup \p \Omega^L_{a,b,\theta_0}$, and a unique point $(b,\theta_2) \in \p \Omega^R_{a,b,\theta_0}$ with $r_1<b$, $\theta_1>0$, $\theta_2 \in (0,\theta_0)$, such that $\psi(r_1,\theta_1)=\psi(b, \theta_2)=c$. 
    By similar arguments as in Case (2a), 
    the monotonicity of $\psi$ implies that for any $(r,\theta) \in L_c$, it holds that $r_1<r<b$, and also for any $r_0 \in (r_1,b)$, there exists a unique $\theta_{r_0} \in (0,\theta_0)$ such that $\psi(r_0,\theta_{r_0})=c$. Then the implicit function theorem implies that $L_c$ must be a $C^1$ simple curve, which can be parameterized as $\theta=\theta(r)$, $r\in (r_1,b)$. 

    \medskip

    \textbf{(3b)} $0=a<b<\infty$.

    \medskip

    Since $\p_\theta \psi>0$ in $\overline{\Omega_{0,b,\theta_0}} \setminus \{\bm{0} \}$, it follows that $\psi(b,\theta)$ is strictly increasing in $\theta \in (0,\theta_0)$, and 
    \begin{equation*}
        \lim_{r \to 0^+} \psi(r,\theta_0) \geqslant C. 
    \end{equation*}
    Therefore, the maximum of $\psi$ in $\overline{\Omega_{0,b,\theta_0}} \setminus\{ \bm{0} \}$ can only be attained on $\p \Omega^{TR}_{0,b,\theta_0}$, while the minimum can only be achieved on $\p \Omega^{B}_{0,b,\theta_0}$. The range $J=(C,\psi(\p \Omega^{TR}_{0,b,\theta_0}))$ forms an open interval. 
    
    We decompose $J$ into two disjoint subsets:  
    \begin{equation*}
        J=J_1 \cup J_2,
    \end{equation*}
    where $J_1:=\psi(\p \Omega^{T}_{0,b,\theta_0} \setminus \p \Omega^{TR}_{0,b,\theta_0})$ is nonempty, and $J_2:= J \setminus J_1$. Note that $J_2$ is empty if $\lim_{r \to 0^+} \psi(r,\theta_0)=C$. 

    For any $c \in J$, there exists a unique point $(b,\theta_2) \in \p \Omega^R_{0,b,\theta_0}$ with $0<\theta_2<\theta_0$ satisfying $\psi(b,\theta_2)=c$. 
    
    If $c \in J_1$, there exists a unique point $(r_1,\theta_0) \in \p \Omega^T_{0,b,\theta_0}$ with $0<r_1<b$, such that $\psi(r_1,\theta_0)=\psi(b,\theta_2)=c$. Analogously, we can show that $L_c$ forms a $C^1$ simple curve  parameterized by $\theta=\theta(r)$ for $r\in (r_1,b)$. 

    If $J_2 \neq \emptyset$ and $c \in J_2$, then $\lim_{r \to 0^+} \psi(r,\theta_0) >C$. Similar as the proof of (2c), we deduce that $L_c$ is a $C^1$ curve that admits the parameterization $\theta=\theta(r)$ for $r \in (0, b)$.   
    
    \medskip
    
    \textbf{(3c)} $0<a<b=\infty$.

    \medskip

    In this case, the minimum of $\psi$ in $\overline{\Omega_{a,\infty,\theta_0}}$ can only be reached on $\p \Omega^{B}_{a,\infty,\theta_0}$, and the maximum cannot be achieved in $\overline{\Omega_{a,\infty,\theta_0}}$, which means that $J=(C, \lim_{r \to +\infty} \psi(r,\theta_0))$. 
    
    For any $c \in J$, there exists a unique point $(r_1,\theta_1) \in \p \Omega^T_{a,\infty,\theta_0} \cup \p \Omega^L_{a,\infty,\theta_0}$ with $\theta_1>0$ satisfying $\psi(r_1,\theta_1)=c$. Similar as the proof of (2f), we conclude that $L_c$ forms a $C^1$ simple curve parameterized by $\theta=\theta(r)$ for $ r\in (r_1,+\infty)$. 

    \medskip

    \textbf{(3d)} $0=a<b=\infty$.

    \medskip

    In this case, $\lim_{r \to 0^+} \psi(r,\theta_0) \geqslant C$, and the maximum of $\psi$ in $\overline{\Omega_{0,\infty,\theta_0}} \setminus\{ \bm{0} \}$ cannot be attained in $\overline{\Omega_{0,\infty,\theta_0}} \setminus \{\bm{0} \}$, while the minimum can only be achieved on $\p \Omega^{B}_{0,\infty,\theta_0}$. 
    Therefore, the interval $J=(C,\lim_{r \to +\infty} \psi(r,\theta_0))$. 
    
    We decompose $J$ into two disjoint subsets:  
    \begin{equation*}
        J=J_1 \cup J_2,
    \end{equation*}
    where $J_1:= \psi(\p\Omega_{0,\infty,\theta_0}^T)$ is nonempty, $J_2:=J\setminus J_1$. Note that $J_2$ is empty if $\lim_{r \to 0^+} \psi(r,\theta_0)=C$.
    
    For any $c \in J$, there are two possible cases: 

    If $c \in J_1$, there exist a unique point $(r_1,\theta_0) \in \p \Omega_{0,\infty,\theta_0}^T$  such that $\psi(r_1,\theta_0)=c$. Following a similar argument as before, we deduce that $L_c$ is a $C^1$ simple curve parameterized as $\theta=\theta(r)$ for $r\in (r_1,+\infty)$. 

    If $J_2 \neq \emptyset$ and $c \in J_2$, then $\lim_{r \to 0^+} \psi(r,\theta_0) >C$. Similar as the proof of (2g), we conclude that $L_c$ is a $C^1$ simple curve, which can be represented as $\theta=\theta(r)$ for $r\in (0,+\infty)$.

    \medskip

    \noindent (ii) By symmetry, the conclusion follows from part (i) under the condition $0<a<b<\infty$ upon interchanging $r$ and $\theta$.   This completes the proof of Lemma \ref{lem:2}. 
\end{proof}

Lemma \ref{lem:2} is parallel to Lemma 2.5 in \cite{HamelNadirashvili2017}, Lemma 2.3 in \cite{HamelNadirashvili2019} and Lemma 2.6 in \cite{HamelNadirashvili2023}. 
In fact, the structure of level sets can still be characterized for functions with lower regularity under certain conditions. See, for instance, \cite{BourgainKorobkovKristensen, Elekes} and the references therein. 

The third lemma addresses the functional dependent property between the stream function $\psi$ and its Laplacian. Specifically, we prove that $\psi$ satisfies an elliptic equation.

\begin{lemma}{\label{lem:3}}
    Let $0\leqslant a<b\leqslant \infty$, $\theta_0\in (0, 2\pi]$, $\Omega_{a,b,\theta_0}$ be defined by $(\ref{eq:domain})$. Let $\vu=u_\theta \bm{e_\theta}+u_r \bm{e_r} \in C^2(\overline{\Omega_{a,b,\theta_0}} \setminus \{\bm{0} \} )$ be a solution of \eqref{eq:euler}, and $\psi$ be a stream function of $\vu$ satisfying one of  the conditions (i) and (ii) in Lemma \ref{lem:2}. Then there exists a $C^1$ scalar function $g: \psi(\Omega_{a,b,\theta_0}) \to \mathbb{R}$, such that 
    \begin{equation*}
        \Delta \psi= g(\psi) \quad \text{in}\ \Omega_{a,b,\theta_0}. 
    \end{equation*}
\end{lemma}
\begin{proof}
    Denote $J:=\psi(\Omega_{a,b,\theta_0})$. Since $\vu \in C^2(\overline{\Omega_{a,b,\theta_0}} \setminus \{\bm{0} \})$ is a solution of \eqref{eq:euler}, we have 
    \begin{equation*}
    \left\{
    \begin{aligned}
        & u_r \frac{\p u_r}{\p r}+\frac{u_\theta}{r} \frac{\p u_r}{\p \theta}-\frac{u_\theta^2}{r}+\frac{\p P}{\p r} =0, \\
        &  u_r \frac{\p u_\theta}{\p r}+\frac{u_\theta}{r} \frac{\p u_\theta}{\p \theta}+\frac{u_\theta u_r}{r}+\frac{1}{r} \frac{\p P}{\p \theta}=0. 
    \end{aligned}
    \right.
    \end{equation*}
    Substitute $u_\theta=\frac{\p \psi}{\p r}$ and $u_r=-\frac{1}{r} \frac{\p \psi}{\p \theta}$ in the above, we have 
    \begin{equation}\label{eq:3.3}
        \begin{cases}
        \begin{aligned}
            &\frac{1}{r^2} \frac{\p \psi}{\p \theta} \frac{\p^2 \psi}{\p r \p \theta}-\frac{1}{r^2}  \frac{\p \psi}{\p r} \frac{\p^2 \psi}{\p \theta^2}-\frac{1}{r^3} \bigg(\frac{\p \psi}{\p \theta} \bigg)^2 -\frac{1}{r} \bigg(\frac{\p \psi}{\p r} \bigg)^2+\frac{\p P}{\p r} =0,  \\
            &\frac{1}{r} \frac{\p \psi}{\p r} \frac{\p^2 \psi}{\p r \p \theta}-\frac{1}{r} \frac{\p \psi}{\p \theta} \frac{\p^2 \psi}{\p r^2}-\frac{1}{r^2} \frac{\p \psi}{\p r} \frac{\p \psi}{\p \theta}+\frac{1}{r} \frac{\p P}{\p \theta} =0.
        \end{aligned}
        \end{cases}
    \end{equation}
    Multiply the second line of \eqref{eq:3.3} by $r$ and differentiate with respect to $r$, then minus the derivative with respect to $\theta$ of the first line of \eqref{eq:3.3}, we obtain the Jacobian determinant
    \begin{equation}\label{eq:3_J}
        \bigg| \frac{\p (\Delta \psi,\psi)}{\p (r,\theta)} \bigg| \equiv 0, \quad (r,\theta) \in \Omega_{a,b,\theta_0}. 
    \end{equation}

    Assume $\psi$ satisfies one of the conditions in Lemma \ref{lem:2}. We only prove the lemma under condition (i) in Lemma \ref{lem:2}. The proof under condition (ii) is similar.
    
    Assume (i) in Lemma \ref{lem:2} holds, we have $|\p_\theta \psi|>0$ in $\Omega_{a,b,\theta_0}$.
    Then by Lemma \ref{lem:2}, the level set $L_c$ for $c \in J$ is a $C^1$ simple curve, which is connected and can be parametrized as $(r,\theta(r))$ for $r$ belongs to an open interval $I_c$. Define
    \begin{equation*}
        F(r)=\Delta \psi(r,\theta(r)),\quad r\in I_{c}.
    \end{equation*}
    Since $\psi(r, \theta(r))=c$ on $I_c$, we have $\partial_r\psi+\partial_{\theta}\psi\cdot\theta'(r)=0$ and  thus $\theta'(r)=-\frac{\partial_r\psi}{\partial_{\theta}\psi}$. By (\ref{eq:3_J}), we have
    \begin{equation*}
        F'(r)=\p_r \Delta\psi+\p_\theta \Delta\psi \cdot \theta'(r)=\p_r \Delta\psi+\p_\theta \Delta\psi \cdot \bigg( -\frac{\p_r \psi}{\p_\theta \psi} \bigg) \equiv 0.
    \end{equation*}
    So $\Delta \psi$ is a constant along $L_c$.
    
    By assumption, we have $|\nabla\psi|\geqslant |\frac{1}{r}\partial_{\theta}\psi|>0$ in $\Omega_{a, b, \theta_0}$. By (\ref{eq:3_J}) and this, using Lemma \ref{lem:1}, for any $\bm{p}\in \Omega_{a, b, \theta_0}$, there exists an open neighborhood $U_{\bm{p}}\subset\Omega_{a, b, \theta_0} $ of $\bm{p}$ and a function $g_{\bm{p}}\in C^1(\psi(U_{\bm{p}}), \mathbb{R})$, such that 
    \[
        \Delta\psi=g_{\bm{p}}(\psi), \quad (r, \theta)\in U_{\bm{p}}. 
    \]
    Since $|\partial_{\theta}\psi|>0$, $\psi(r, \theta)$ is monotone in $\theta$ 
    for each fixed $r$. So for each $r$, $\psi(U_{\bm{p}}\cap\{|x|=r\})$ is a nonempty open interval, and thus $\psi(U_{\bm{p}})$ is a nonempty open interval. 
    Then for any $c\in J$ and $\bm{p}, \bm{q} \in L_c$,  the intersection $\psi(U_{\bm{p}}) \cap \psi(U_{\bm{q}})$ is a non-empty open interval. For any $c' \in \psi(U_{\bm{p}}) \cap \psi(U_{\bm{q}})$, we have $\Delta \psi$ is constant along $L_{c'}$. It follows that $g_{\bm{p}}(c')=g_{\bm{q}}(c')$. Therefore, $g_{\bm{p}}\equiv g_{\bm{q}}$ on $\psi(U_{\bm{p}}) \cap \psi(U_{\bm{q}})$. Now for $c\in J$, we define 
    \[
        g(c):=g_{\bm{p}}(c), \quad \bm{p}\in L_c.
    \]
    Then $g(c)$ is well-defined for each $c$. Moreover, for any fixed $p\in L_c$,  there exists some $\epsilon>0$, such that $(c-\epsilon, c+\epsilon)\subset \psi(U_{\bm{p}})$. Note $g$ is $C^1$ in $\psi(U_{\bm{p}})$, we have $g\in C^1(c-\epsilon, c+\epsilon)$. So $g\in C^1(J)$. 
    The proof of the lemma is finished. 
\end{proof}

\begin{remark}
    The above lemma addresses the functional dependence of $\psi$ and $\Delta \psi$ in the interior domain. 
    By the continuity of $\psi$ and $\Delta \psi$ up to the boundary, the function $g$ can be extended to $I:=\psi(\overline{\Omega_{a,b,\theta_0}}\setminus \{\bm{0} \})$ in such a way that $g \in C^1(J)\cap C(I)$ and $\Delta \psi= g(\psi)\ {\rm{ in }}\ \overline{\Omega_{a,b,\theta_0}}\setminus \{\bm{0} \}$. 
\end{remark}

Lemma \ref{lem:3} can be viewed as a parallel result, under some different assumptions, as Lemma 2.1 in \cite{ConstantinDrivasGinsberg2021}, Lemma 2.8 in \cite{HamelNadirashvili2019} and Lemma 2.8 in \cite{HamelNadirashvili2023}, when the domain is an annulus, a half-space, a strip or the entire space. 

Next, we establish a Liouville-type result for some elliptic equation. 
Let $L=a^{ij}\p_{ij}+b^i\p_i+C$ denote a uniformly elliptic operator with constant coefficients $a^{ij}, b^i,C$, $i,j=1,2$, such that 
\[
    a^{ij}\xi_i\xi_j\geqslant \lambda |\xi|^2, \quad \forall\xi\in \mathbb{R}^2,
\]
for some $\lambda>0$.
\begin{lemma}\label{lem:4}
    Let $-\infty \leqslant x_1<x_2 \leqslant +\infty$, $-\infty<y_1<y_2<+\infty$, $\Omega=(x_1,x_2) \times (y_1,y_2)\subset \mathbb{R}^2$, $L=a^{ij}\p_{ij}+b^i\p_i+C$ be a uniformly elliptic operator with constant coefficients $a^{ij}, b^i,C$, $i,j=1,2$. Let $c_1\ne c_2$, $F(y) \in C^1([y_1,y_2])$ and $g: \overline{Range(\psi)} \to \mathbb{R}$ be a Lipschitz function, satisfying that either $F(y) \equiv {\rm{ const }}$ or $F'(y)\geqslant 0$ and $(c_2-c_1)g(y)\leqslant 0$. Assume $\psi \in C^2(\overline{\Omega})$ satisfies 
    \begin{equation}\label{eq4_1}
    \left\{
    \begin{aligned}
        &  L\psi=F(y)g(\psi), \quad (x,y)\in \Omega,\\
        & \psi(x,y_1)=c_1, \quad \psi(x,y_2)=c_2, \quad (c_2-c_1)\psi_y\geqslant 0, 
    \end{aligned}
    \right.
    \end{equation}
    and one of the following holds:

    {\rm (i)} $x_1=-\infty$, $x_2=+\infty$.

    {\rm (ii)} $-\infty<x_1<x_2<+\infty$, $\partial_x^j\psi(x_1, y)=\partial_x^j\psi(x_2, y)$ for $0\leqslant j\leqslant 2$. 

    {\rm (iii)} $-\infty=x_1< x_2<+\infty$, $b^1=b^2 = 0$, and $\partial_x\psi(x_2,y) =0$. 

    {\rm (iv)} $-\infty<x_1< x_2=+\infty$, $b^1=b^2=0$, and $\partial_x\psi(x_1,y) =0$.
 
    \noindent
    Then $\psi$ is independent of $x$, namely, $\psi(x,y)=\psi(y)$.
     
\end{lemma}
\begin{proof}  
    Without loss of generality, assume $y_1=0$, $y_2=1$, and set $c_1=0$, $c_2=c>0$ by shifting $\psi \mapsto \psi-c_1$. If $c_2<c_1$, we could replace $\psi$ by $-\psi$ and $g(\psi)$ by $-g(-\psi)$. 

    By assumption, when $c_2>0=c_1$, we have $\psi_y \geqslant 0$, and either $F\equiv const$ or $F'\geqslant 0$ and $g\leqslant 0$.

    \medskip

    \noindent Step 1. We first extend $\psi$ into a function $\Psi$ define on $ D:=\mathbb{R} \times [0,1]$. 

    If $-\infty=x_1<x_2=+\infty$, we simply set $\Psi(x,y)=\psi(x,y)$ for all $(x,y) \in D$. 

    If $-\infty<x_1<x_2 <+\infty$, we define the periodic extension: 
    \begin{equation*}
        \Psi(x,y)=\psi(x-k(x_2-x_1),y), \quad x \in \mathbb{R},\ y\in [0,1].  
    \end{equation*}
    where $k=\lfloor \frac{x-x_1}{x_2-x_1} \rfloor$ denotes  the greatest integer less than or equal to $\frac{x-x_1}{x_2-x_1}$. In this case, $\Psi$ is periodic in $x$ with period $x_2-x_1$. 

    If $-\infty=x_1<x_2 <+\infty$, we define $\Psi$ to be the even extension of $\psi$: 
    \begin{equation*}
        \Psi(x,y)=
        \begin{cases}
            \psi(x,y),  &\quad x \leqslant x_2, \\
            \psi(2x_2-x,y), &\quad x> x_2.
        \end{cases}
    \end{equation*}

    If $-\infty<x_1< x_2=+\infty$, we apply an analogous even extension about $x=x_1$. 

    By assumptions (i)-(iv), we have $\Psi \in C^2(\overline{D})$ and  %is a solution to 
    \begin{equation}\label{eq4_Psi}
    \left\{
    \begin{aligned}
        & L \Psi=F(y) g(\Psi), \quad   (x,y)\in D,\\
        & \Psi(x,0)=0, \quad \Psi(x,1)=c>0, \quad \Psi_y \geqslant 0\textrm{ in }D.
    \end{aligned}
    \right.
    \end{equation}
    Note that this implies $0 \leqslant \Psi \leqslant c$ in $D$. 

    \medskip

    \noindent Step 2. Now we prove that $\Psi$ is strictly increasing with respect to $y$ and thus $0<\Psi<c$ in $D$. 

    For any integer $k \geqslant 2$, define 
    \begin{equation*}
        \varphi(x,y):=\Psi(x,y+\frac{1}{k})-\Psi(x,y), \quad (x,y)\in \mathbb{R}\times [0,1-\frac{1}{k}]. 
    \end{equation*}
    For convenience, denote $D_k:=\mathbb{R}\times (0, 1-\frac{1}{k})$.  Since $\Psi_y \geqslant 0$ in $D$, we have  $\varphi\geqslant 0$ in $D_k$.
    For $(x, y)\in D_k$, let
    \begin{equation*}
        \tilde{c}(x,y):=
        \left\{
        \begin{aligned}
            &\frac{g(\Psi(x,y+\frac{1}{k}))-g(\Psi(x,y))}{\Psi(x,y+\frac{1}{k})-\Psi(x,y)}, && \text{ if } \Psi(x,y+\frac{1}{k})\neq \Psi(x,y), \\
            & 0,  && \text{ if } \Psi(x,y+\frac{1}{k})=\Psi(x,y).
        \end{aligned}
        \right.
    \end{equation*}
    Since $g$ is locally Lipschitz and $\Psi$ is bounded, it follows that the function $\tilde{c}(x,y)$ is in $ L^\infty(D_k)$. Let $\tilde{L}:=L-F(y)\tilde{c}(x,y)$. 
    By assumption, we have either $F=const$ or $F'\geqslant 0$ and $g\leqslant 0$.  By this and the first equation in (\ref{eq4_Psi}), it is easy to check that
    \begin{equation*}
        \tilde{L}\varphi=[F(y+\frac{1}{k})-F(y)]g(\Psi(x,y+\frac{1}{k})) \leqslant 0, \quad (x,y)\in D_k. 
    \end{equation*}
    By the maximum principle for  nonnegative functions (see, e.g. Theorem 2.10 in \cite{HanLin}), we have either $\varphi>0$ or $\varphi\equiv 0$ on $D_k$, i.e. either 
    \begin{equation}\label{lem4:eq10}
        \Psi(x,y+\frac{1}{k})=\Psi(x,y), \quad \forall (x,y) \in \mathbb{R} \times [0,1-\frac{1}{k}],
    \end{equation}
    or 
    \begin{equation}\label{lem4:eq20}
        \Psi(x,y+\frac{1}{k})>\Psi(x,y), \quad \forall (x,y) \in \mathbb{R} \times (0,1-\frac{1}{k}).
    \end{equation}
    
    However, the case (\ref{lem4:eq10}) would imply that $0=\Psi(x,0)=\Psi(x,\frac{1}{k})=\cdots=\Psi(x,1)=c$, which contradicts the assumption that $c>0$. Therefore, (\ref{lem4:eq20}) holds for all $k\geqslant 2$. Hence, by choosing $k$ sufficiently large, $\Psi$ is actually strictly increasing with respect to $y$ and satisfies $0<\Psi<c$ in $D$. 

    \medskip

    \noindent Step 3. Fix a vector $\bm{\xi}=(\xi_1,\xi_2) \in \mathbb{R}^2$ with $\xi_1 \in \mathbb{R}$ and $\xi_2>0$. 
    For any $\tau \in (0, 1/\xi_2)$, we define
    \begin{equation*}
        D^{\tau}:=\mathbb{R} \times (0,1-\tau \xi_2), 
    \end{equation*}
    and
    \begin{equation*}
        w^{\tau}(x,y):=\Psi(x+\tau \xi_1, y+\tau \xi_2)-\Psi(x,y),\quad (x,y)\in \overline{D^{\tau}}.  
    \end{equation*}

    \noindent
    {\it Claim:} $w^{\tau}>0$ in $\overline{D^{\tau}}$ for any $\tau \in (0,1/\xi_2)$.
    
    \noindent\emph{Proof of Claim}: Set
    \begin{equation*}
        \mathscr{A} :=\bigg\{ \tau \in \bigg(0,\frac{1}{\xi_2} \bigg): w^{{\tau}'}>0  \text{ in } \overline{D^{\tau'}}, 
        {\rm{for\ all}}\ \tau' \in \bigg(\tau,\frac{1}{\xi_2} \bigg)
        \bigg\}.
    \end{equation*}
    To prove the claim, we show $\mathscr{A}=(0,1/\xi_2)$. We first show that $\mathscr{A}$ is not empty. Note that $\Psi$ satisfies $\Psi(x,0)=0$, $\Psi(x,1)=c>0$ and $0<\Psi<c$ in $D$. For $(x, y)\in D^{\tau}$, let 
    \[
        \hat{c}(x,y):=
        \left\{
        \begin{aligned}
            &\frac{g(\Psi(x+\tau \xi_1,y+\tau \xi_2))-g(\Psi(x,y))}{\Psi(x+\tau \xi_1,y+\tau \xi_2)-\Psi(x,y)}, && \text{ if } \Psi(x+\tau \xi_1,y+\tau \xi_2)\neq \Psi(x,y), \\
            & 0,  && \text{ if } \Psi(x+\tau \xi_1,y+\tau \xi_2)=\Psi(x,y).
        \end{aligned}
        \right.
    \]
    Then $\hat{c}(x,y;\tau)$ is uniformly bounded  in $D^{\tau}$. Let $\hat{L}:=L-F(y)\hat{c}(x,y;\tau)$.
    By definition of $w^\tau$ and (\ref{eq4_Psi}), note we have either $F=const$ or $F'\geqslant 0$ and $g\leqslant 0$,   for any $\tau \in (0,1/\xi_2)$, it can be checked that
    \begin{equation*}
        \begin{cases}
        \hat{L} w^\tau= [F(y+\tau \xi_2)-F(y)]g(\Psi(x+\tau \xi_1,y+\tau \xi_2)) \leqslant 0, \quad (x,y) \in D^{\tau}, \\
        w^\tau|_{\p D^\tau}>0. 
        \end{cases}
    \end{equation*}
    Using the maximum principle for narrow domains (see, e.g. Proposition 2.13 in \cite{HanLin} and Corollary below Lemma H in \cite{GidasNiNirenberg1979}), 
    there exists a suitably small constant $0<\varepsilon <<1$ such that for all $\tau \in (1/\xi_2-\varepsilon,1/\xi_2)$, $D^{\tau}$ is a narrow domain and $w^{\tau}>0$ in $\overline{D^{\tau}}$. Therefore, $\mathscr{A}$ is a nonempty set, and $(1/\xi_2-\varepsilon,1/\xi_2)\subset \mathscr{A}$.
    
    Let $\tau_*:=\inf \mathscr{A}< 1/\xi_2$, we now prove that $\tau_*=0$ and thus the claim holds. 
    Suppose on the contrary that $\tau_*>0$. By definition of $\tau_*$ and the continuity of $\Psi$, we have
    \begin{equation}\label{eq:4.1}
        w^{\tau_*}(x,y) \geqslant 0 \text{ in } \overline{D^{\tau_*}},
    \end{equation}
    and there exist sequences $\tau_k \in (0,\tau_*]$ and $(x_k,y_k) \in \overline{D^{\tau_k}}$, such that 
    $\tau_k \nearrow \tau_*$ and 
    \begin{equation}\label{eq:4.2}
        w^{\tau_k} (x_k,y_k) \leqslant 0.
    \end{equation}
    Define 
    \begin{equation*}
        \Psi_k(x,y):= \Psi(x+x_k,y), \quad (x,y) \in \overline{D}.
    \end{equation*}
    Since $\Psi \in C^2(\overline{D})$ is bounded, $F(y) \in C^1([y_1,y_2])$, and $g$ is Lipschitz on the range of $\Psi$, the function $h(x,y):=F(y)g(\Psi(x,y))$ is bounded and locally Lipschitz in $\overline{D}$. By (\ref{eq4_Psi}), using standard elliptic estimates for $\Psi$ up to the boundary, we know that $\nabla \Psi$ is bounded in $\overline{D}$, and $\Psi\in C^{2,\alpha}_{\text{loc}}(\overline{D})$ for any $\alpha\in [0, 1)$. 
    Thus $\{ \Psi_k \}$ is uniformly bounded in $C^{2,\alpha}_{\text{loc}}(\overline{D})$ for any $\alpha \in [0,1)$. So there exists a subsequence, still denoted as $\{ \Psi_k \}$, and some limit function $\Phi \in C^2(\overline{D})$, such that $\Psi_k\to \Phi$ in $C^{2}_{\text{loc}}(\overline{D})$. Note $\Psi_k$ satisfies (\ref{eq4_Psi}). 
    Taking $k \to \infty$ in (\ref{eq4_Psi}) with $\Psi$ replaced by $\Psi_k$, we have 
    \begin{equation}\label{eq4_Phi}
        \left\{
        \begin{aligned}
            & L \Phi=F(y) g(\Phi), \quad  \text{ in } D, \\
            & \Phi(x,0)=0, \quad \Phi(x,1)=c, \quad \partial_y\Phi \geqslant 0\textrm{ in }D. 
        \end{aligned}
        \right.
    \end{equation}
    By similar arguments to that in Step 2 to $\Phi$ shows that $0<\Phi<c$ in $D$.

    Next, since $\{ y_k \}$ is bounded, there exists a subsequence, still denoted as $\{ y_k \}$, and some $y_*\in [0,1]$, such $y_k \to y_*$ as $k\to \infty$. Since $\overline{D^{\tau_k}} \to \overline{D^{\tau_*}}$, we conclude that $(0,y_*) \in \overline{D^{\tau_*}}$. By (\ref{eq:4.2}) we have 
    \[
        0\geqslant w^{\tau_k} (x_k,y_k)=\Psi(x_k+\tau_k\xi_1, y_k+\tau_k\xi_2)-\Psi(x_k, y_k)=\Psi_k(\tau_k\xi_1, y_k+\tau_k\xi_2)-\Psi_k(0, y_k). 
    \]
    Sending $k\to \infty$ in the above, we have 
    \[
        \Phi(\tau_*\xi_1, y_*+\tau_*\xi_2)-\Phi(0, y_*)\leqslant 0.
    \]
    On the other hand, by (\ref{eq:4.1}), for any $x\in \mathbb{R}$, $y\in \overline{D^{\tau_*}}$, we have 
    \begin{equation}\label{eq4_w_1}
        0\leqslant w^{\tau_*}(x+x_k, y)=\Psi(x+x_k+\tau_*\xi_1, y+\tau_*\xi_2)-\Psi(x+x_k, y)=\Psi_k(x+\tau_*\xi_1, y+\tau_*\xi_2)-\Psi_k(x,y).
    \end{equation}
    Let $x=0$, $y=y_*$ in the above, and send $k\to \infty$ in the above, we have 
    \[
        \Phi(\tau_*\xi_1, y_*+\tau_*\xi_2)-\Phi(0, y_*)\geqslant 0.
    \]
    So we have 
    \begin{equation}\label{eq:4.3}
        \Phi(\tau_*\xi_1, y_*+\tau_*\xi_2)-\Phi(0, y_*)=0. 
    \end{equation}

    We now prove that $(0,y_*)$ must lie in the interior $D^{\tau_*}$. Assume for contradiction that $(0,y_*) \in \p D^{\tau_*}$, then either $y_*=0$ or $y_*=1-\tau_* \xi_2$. Consider first the case $y_*=0$. From (\ref{eq4_Phi}), we have $\Phi(0, 0)=0$. Using (\ref{eq:4.3}) and the facts that $0<\tau_*<1/\xi_2$ and $0<\Phi<c \text{ in } D$, we derive the contradiction: 
    \begin{equation*}
        0=\Phi(0, 0)=\Phi(\tau_*\xi_1, \tau_* \xi_2)>0. 
    \end{equation*}
    The case $y_*=1-\tau_* \xi_2$ similarly leads to a contradiction. So $(0,y_*) \notin \p D^{\tau_*}$ and we must have $(0,y_*) \in D^{\tau_*}$.
    
    Next, set 
    \begin{equation*}
        W(x, y):= \Phi(x+\tau_*\xi_1,y+\tau_*\xi_2)-\Phi(x,y), \quad (x,y) \in \overline{D^{\tau_*}}. 
    \end{equation*}
    Let
    \[
        c'(x,y):=
        \left\{
        \begin{aligned}
            &\frac{g(\Phi(x+\tau_*\xi_1,y+\tau_*\xi_2))-g(\Phi(x,y))}{\Phi(x+\tau_*\xi_1,y+\tau_*\xi_2)-\Phi(x,y)}, && \text{ if } \Phi(x+\tau_*\xi_1,y+\tau_*\xi_2)\ne \Phi(x,y), \\
            & 0,  && \text{ if } \Phi(x+\tau_*\xi_1,y+\tau_*\xi_2)=\Phi(x,y).
        \end{aligned}
        \right.
    \]
    Then $c'(x, y)\in L^{\infty}(\overline{D^{\tau_*}})$. Let $L' :=L-F(y) c'(x,y)$. We know that $L'$ is uniformly elliptic. By assumption, we have either $F=const$ or $F'\geqslant 0$ and $g\leqslant 0$. By this and (\ref{eq4_Phi}), we have 
    \begin{equation*}
        L'W=[F(y+\tau_*\xi_2)-F(y)]g(\Phi_*) \leqslant 0 \quad 
        \text{ in } D^{\tau_*},
    \end{equation*}
    Send $k\to \infty$ in (\ref{eq4_w_1}), we have $W \geqslant 0 \text{ in } \overline{D^{\tau_*}}$. By (\ref{eq:4.3}), we have $W(0,y_*)=0$. Note $(0,y_*) \in D^{\tau_*}$ is an interior point. 
    By the maximum principle for non-negative functions, we have $W \equiv 0 \text{ in } \overline{D^{\tau_*}}$, including on $\p D^{\tau_*}$. 
    On the other hand, since $\Phi(x,0)=0$ and $0<\Phi<c$ in $D$, we have $W(x,0)>0$. This leads to a contradiction. Thus, we conclude that $\tau_*=0$.
         
    \medskip
    
    \noindent Step 4. We now prove that $\psi$ is independent of $x$. 
    
    \medskip
    
    By Step 3, for any fixed $\xi_2>0$, $\tau \in (0,1/\xi_2)$ and $\xi_1 \in \mathbb{R}$, we have 
    \begin{equation*}
        \Psi(x+\tau \xi_1, y+\tau\xi_2)-\Psi(x,y)= w^{\tau}(x, y)>0  \text{ in } \overline{D^{\tau}}.
    \end{equation*}
    Note for $0<\xi_2<1$, the interval $(0,1) \subseteq (0,1/\xi_2)$. Sending $\xi_2 \to 0^+$ in the above, we see that for all $\tau \in (0,1)$, $(x,y) \in D$ and $\xi_1 \in \mathbb{R}$, 
    \begin{equation*}
        \Psi(x+\tau \xi_1, y) \geqslant \Psi(x,y).
    \end{equation*}
    Since $\xi_1$ and $x$ are arbitrary, 
    this implies that $\Psi(x+\tau \xi_1, y)= \Psi(x,y)$ for any $\xi$ and thus $\Psi(x,y)$ is independent of $x$. 
    As a consequence, $\psi$ is independent of $x$. 
    This completes the proof of the lemma.
\end{proof}

\begin{remark}
    For the case $-\infty<x_1<x_2<+\infty$, the condition that $(c_2-c_1)\psi_y \geqslant 0$ can be replaced by an alternative assumption that $\psi$ strictly lies between $c_1$ and $c_2$ for any $(x,y) \in \overline{\Omega}$ with $y_1<y<y_2$. If this assumption is imposed, then Step 2 can be omitted, and in Step 3, it is possible to directly set $x_k \to x_*$ so that $\Psi_k(x,y)$ converges directly to $\Phi(x,y)=\Psi(x+x_*, y)$, which satisfies $\Phi(x,0)=0$, $\Phi(x,1)=c$ and $0<\Phi<c$ in $D$. Then, the same subsequent argument can be carried out. 
\end{remark}

\begin{remark}
    If $L=\Delta$ and $F \equiv 1$, the condition that $(c_2-c_1)\psi_y \geqslant 0$ can still be replaced by an alternative assumption that $\psi$ strictly lies between $c_1$ and $c_2$ for any $(x,y) \in \overline{\Omega}$ with $y_1<y<y_2$, for all four types of domains. The result for strip can refer to Theorem 1.4 and Remark 1.5 in \cite{HamelNadirashvili2017}, as well as Theorem 1.1 and Theorem 1.1' in \cite{BerestyckiCaffarelliNirenberg1997}. 
\end{remark}

\section{Proof of main results}

We present in this section the proofs of Theorems \ref{thm:1}-\ref{thm:5}. Let $\psi$ be a stream function for the velocity field $\vu=u_\theta \bm{e_\theta}+u_r \bm{e_r} \in C^2(\overline{\Omega_{a,b,\theta_0}} \setminus \{ \bm{0} \})$. Then $\psi \in C^3(\overline{\Omega_{a,b,\theta_0}} \setminus \{ \bm{0} \})$ and satisfies (\ref{thm0:eq0}). 

\begin{lemma}\label{lem1_1}
    Let $0<\theta_0\leqslant 2\pi$, $\alpha\in \mathbb{R}$, $\vu \in C^1(\overline{\Omega_{1,2,\theta_0}})$, ${\rm div}\, \vu = 0$, 
    satisfying (\ref{eq:BC:hom:-alpha:theta}) for some constants $c_1, c_2$, and 
    \begin{equation}\label{eq1_1_1}
        u_r (1, \theta)=2^{\alpha}u_r(2, \theta), \quad \theta\in [0, \theta_0].
    \end{equation}
    Then the stream function $\psi$ of $\vu$ satisfies the following on $\partial\Omega_{1, 2, \theta_0}$. 

    \noindent (i) If $\alpha=1$, then $c_1=c_2=:c$, and 
    \begin{equation}\label{eq1_1_2}
        \psi(r, \theta)=c\ln r-\int_0^{\theta}u_r(1, \varphi)d\varphi+A, \quad \textrm{ on }\partial\Omega_{1, 2, \theta_0},
    \end{equation}
    for some constant $A$.

    \noindent (ii) If $\alpha\ne 1$, then 
    \begin{equation}\label{eq1_1_3}
        \psi=r^{1-\alpha}( - \int_0^{\theta}u_r(1, \varphi)d\varphi+\frac{c_1}{1-\alpha}) + A, \quad \textrm{ on }\partial\Omega_{1, 2, \theta_0}, 
    \end{equation}
    for some constant $A$.
\end{lemma}
\begin{proof}
    By (\ref{eq:BC:hom:-alpha:theta}) and (\ref{thm0:eq0}), we have 
    \begin{equation}\label{eq1_1_4}
    \partial_r\psi(r, 0)=u_\theta|_{\theta=0}=\frac{c_1}{r^\alpha}, \quad \partial_r\psi(r, \theta_0)=u_\theta|_{\theta=\theta_0}=\frac{c_2}{r^\alpha}.
    \end{equation}
    (i) If $\alpha=1$, then by the above, we have 
       \begin{equation}\label{eq1_1_5}
        \psi(r,0) =c_1 \ln r+A_1, \quad 
        \psi(r,\theta_0) =c_2 \ln r+A_2, \quad r\in [1,2],
    \end{equation}
    for some constants $A_1, A_2$. 
    By (\ref{thm0:eq0}), we also have $\partial_{\theta}\psi=-ru_r$. 
    By this and (\ref{eq1_1_1}) with $\alpha=1$, we have  $\partial_{\theta}\psi(1, \theta)=-u_r|_{r=1}=-2 u_r|_{r=2}=\partial_{\theta}\psi(2, \theta)$.  This and the above imply 
    \begin{equation}\label{eq1_1_6}
        \psi(1,\theta) 
        =h(\theta)+A_1,
        \quad \psi(2, \theta)=
        h(\theta)+c_1\ln 2+A_1, \quad \theta\in [0, \theta_0], 
    \end{equation}
    where $h(\theta):= -\int_0^\theta u_r(1,\varphi) \text{d} \varphi\in C^3([0,\theta_0])$.  
    Note (\ref{eq1_1_5}) and (\ref{eq1_1_6}) imply
    \begin{equation*}
        \begin{cases}
            \psi(1, \theta_0)=A_2=h(\theta_0)+A_1, \\
            \psi(2, \theta_0)=c_2 \ln 2+A_2=h(\theta_0)+c_1\ln 2+A_1. 
        \end{cases}
    \end{equation*}
    This implies $c_1=c_2=:c$. Thus we have 
    \begin{equation*}
        \psi(r,\theta)=c \ln r+h(\theta) + A_1, \quad \text{on } \p \Omega_{1,2,\theta_0}.
    \end{equation*}
    So (i) is proved.

    \noindent (ii) If $\alpha\ne 1$, by (\ref{eq1_1_4}), we have 
    \begin{equation}\label{eq1_1_7}
        \psi(r,0)=\frac{c_1}{1-\alpha} r^{1-\alpha}+A_1, \quad 
        \psi(r,\theta_0)=\frac{c_2}{1-\alpha} r^{1-\alpha}+A_2, \quad r \in [1,2]. 
    \end{equation}
    Note (\ref{thm0:eq0}) says $\partial_{\theta}\psi=-ru_r$.  By this and (\ref{eq1_1_1}), we have
    \begin{equation}\label{eq1_1_8}
        \psi(1,\theta)=h(\theta)+A_1, \quad \psi(2,\theta)=2^{1-\alpha} h(\theta)+A_1, \quad \theta \in [0,\theta_0],
    \end{equation}
    where $h(\theta):= -\int_0^\theta u_r(1,\varphi) \text{d} \varphi+\frac{c_1}{1-\alpha}\in C^3([0,\theta_0])$. 
    By (\ref{eq1_1_7}) and (\ref{eq1_1_8}), we have 
    \begin{equation*}
    \begin{cases}
        \psi(1, \theta_0)=\frac{c_2}{1-\alpha} +A_2=h(\theta_0)+A_1, \\
        \psi(2, \theta_0)=2^{1-\alpha}\frac{c_2}{1-\alpha} +A_2=2^{1-\alpha} h(\theta_0)+A_1. 
    \end{cases}
    \end{equation*}
    Since $\alpha\ne 1$, the above implies $c_2=(1-\alpha)h(\theta_0)$ and $A_2=A_1=:A$. Then we have
    \begin{equation*}
        \psi(r,\theta)=h(\theta) r^{1-\alpha} + A, \quad (r,\theta)\in \p \Omega_{1,2,\theta_0}. 
    \end{equation*}
    So (ii) is proved. 
\end{proof}

We will use the following basic facts. 
\begin{lemma}\label{lem1_2}
    Let $a < b$, $P > 0$, and set $J:=(a,b+P)$. Suppose $G\in C(J)$ satisfies $G(x+P)=G(x)$ for all $x\in (a,b)$, and $G(x)\equiv C$ on $x\in (a,a+P)$ or $x\in (b,b+P)$. Then $G\equiv C$ on $J$. 
\end{lemma}

\begin{lemma}\label{lem1_3}
    Let $0\leqslant a < b$ or $a < b \leqslant 0$, $\lambda > 0$, $\beta\in\mathbb{R}$. Set $J:=(\min\{\lambda^{-1} a, a\}, \max\{b, \lambda^{-1} b\})$. Suppose $G\in C(J)$ satisfies $G(x) = \lambda^\beta G(\lambda^{-1} x)$ for all $x\in (a,b)$, and $G(x) = C |x|^{\beta}$ on $x\in ( \min \{ \lambda^{-1} b, b \}, \max \{ \lambda^{-1} b, b \} )$ when $b \not= 0$ or $x\in ( \min \{ \lambda^{-1} a, a \}, \max \{ \lambda^{-1} a, a \} )$ when $a \not= 0$. Then $G = C|x|^{\beta}$ on $J$. 
\end{lemma}

We now proceed with the proof of Theorem \ref{thm:1}. 

\begin{proof}[Proof of Theorem \ref{thm:1}]

    1. Let $J:=\{\psi(r,\theta):(r,\theta) \in \Omega_{1,2,\theta_0} \}$. We show $\psi$ satisfies $\Delta\psi=g(\psi)$ for some $g\in C^1(J)$ and identify $\psi$ on $\partial\Omega_{1, 2, \theta_0}$. 
    By (\ref{eq:BC:hom:-alpha:theta}), (\ref{thm0:eq0}) and Lemma \ref{lem1_1} with $\alpha=1$, we have $c_1=c_2=c$ and (\ref{eq1_1_2}). Without loss of generality, let $A=0$ in  (\ref{eq1_1_2}). Then
    \begin{equation}\label{thm1:eq2}
        \psi(r,\theta)=c \ln r+h(\theta), \quad \text{on } \p \Omega_{1,2,\theta_0},
    \end{equation}
    where $h(\theta):= -\int_0^\theta u_r(1,\varphi) \text{d} \varphi \in C^3([0,\theta_0])$. By assumption, we have either $u_r>0$ or $u_r<0$ in $\overline{\Omega_{1,2,\theta_0}}$. So $h=-\int_0^\theta u_r(1,\varphi) \text{d} \varphi$ is strictly monotone and $B:=h(\theta_0)\ne 0$. Moreover, the extreme values of $\psi$ are attained on $\p \Omega_{1,2,\theta_0}$. When $c=0$, by (\ref{thm1:eq2}) we have $\psi=h(\theta)$ on $\partial \Omega_{1,2,\theta_0} $, and thus $J=(\min \{0,B \}, \max \{0, B \})$.  When $c\ne 0$, by (\ref{thm1:eq2}) and the fact that  $h$ is strictly monotone, we have $\psi$ attains its extrema exclusively at two diagonally opposite vertices of $\overline{\Omega_{1,2,\theta_0}}$: either $\{\p \Omega^{TL}_{1,2,\theta_0}, \p \Omega^{BR}_{1,2,\theta_0} \}$ or $ \{ \p \Omega^{TR}_{1,2,\theta_0}, \p \Omega^{BL}_{1,2,\theta_0} \}$.
    By continuity of $\psi$ and the connectivity of $\overline{\Omega_{1,2,\theta_0}}$, the range $J$ forms a bounded open interval excluding these extreme values. Specifically, $J =\psi(\ell)$, where $\ell$ is any open continuous curve in $\overline{\Omega_{1,2,\theta_0}}$ connecting the two extremal points. 
    
    In particular, the above implies  the condition (i) in Lemma \ref{lem:2} holds. Then by Lemma \ref{lem:3}, there exists a scalar function $g \in C^1(J)$, such that
    \begin{equation}\label{thm1:eq4}
        \Delta \psi=\frac{\p^2 \psi}{\p r^2}+\frac{1}{r} \frac{\p \psi}{\p r}+\frac{1}{r^2} \frac{\p^2 \psi}{\p \theta^2}=g(\psi), \quad (r,\theta) \in \Omega_{1,2,\theta_0}.
    \end{equation}
    By (\ref{eq:BC:hom:-alpha:r}) with $\alpha=1$, we have $u_\theta|_{r=1}=2u_\theta|_{r=2}$ and $\p_r u_\theta|_{r=1}=4\p_r u_\theta|_{r=2}$. By this and (\ref{thm0:eq0}), we have
    \begin{equation}\label{thm1:eq5}
        \frac{\p \psi}{\p r}(1,\theta)=
        2\frac{\p \psi}{\p r}(2,\theta),\quad
        \frac{\p^2 \psi}{\p r^2}(1,\theta)=
        4\frac{\p^2 \psi}{\p r^2}(2,\theta).
    \end{equation} 
    Let $s:=\ln r$, $D:=(0,\ln 2) \times (0,\theta_0)$ and 
    \begin{equation}\label{thm1:eq6}
        \Psi(s,\theta) :=\psi(e^s,\theta)-cs, \quad (s,\theta)\in \overline{D}.
    \end{equation}
    Then $\Psi \in C^3(\overline{D})$ and  $|\p_\theta \Psi|=|\p_\theta \psi|>0$. By (\ref{thm1:eq4}) and the above, we have
    \begin{equation}\label{thm1:eq7}
        \frac{\p^2 \Psi}{\p s^2}+\frac{\p^2 \Psi}{\p \theta^2}=e^{2s} g(\Psi+cs), \quad (s,\theta)\in  D.
    \end{equation}
    By (\ref{thm1:eq2}), (\ref{thm1:eq5}) and (\ref{thm1:eq6}), we have 
    \begin{equation}\label{thm1:eq8'}
        \Psi(s,\theta) = h(\theta), \quad \text{ on } \p D,
    \end{equation}
    \begin{equation}\label{thm1:eq8}
        \frac{\p \Psi}{\p s}(0,\theta)= \frac{\p \Psi}{\p s}(\ln 2,\theta),\quad
        \frac{\p^2 \Psi}{\p s^2}(0,\theta) = \frac{\p^2 \Psi}{\p s^2}(\ln 2,\theta), \quad \theta \in [0,\theta_0].
    \end{equation}
    In particular,  
    \begin{equation}\label{thm1:eq_B}
      \Psi(s, 0)=h(0)=0, \quad \Psi(s, \theta_0)=h(\theta_0)=B.
    \end{equation}
    
    Taking limits as $s \to 0^+$ and $s \to (\ln 2)^-$ in (\ref{thm1:eq7}), and applying (\ref{thm1:eq8'})-(\ref{thm1:eq8}), we obtain that 
    \begin{equation}\label{thm1:eq11}
        g(h(\theta))=4g(h(\theta)+c \ln 2), \quad \theta \in (0,\theta_0).
    \end{equation}
    
    \noindent 2. Now we prove Theorem \ref{thm:1} (i), where $c=0$. In this case, we have $J=(\min \{0,B \}, \max \{0, B \})$, where $B=h(\theta_0)$. By (\ref{thm1:eq11}), we have $g(h(\theta))=4g(h(\theta))$ and thus $g(h(\theta))=0$ for $\theta\in (0, \theta_0)$. So $g\equiv 0$ in $J$.
    In view of  (\ref{thm1:eq7}), we apply Lemma \ref{lem:4} to $\Psi$ with $L=\p_{s}^2+\p_{\theta}^2$,  $F=g=0$, and $\Omega=D$ there. Note we have $|\p_\theta \Psi|>0$, (\ref{thm1:eq8'})-(\ref{thm1:eq_B}), and condition (ii) in Lemma \ref{lem:4} holds for $\Psi$. Then by Lemma \ref{lem:4}, we have $\Psi(s,\theta)=h(\theta)$ in $\overline{D}$.
    Thus (\ref{thm1:eq7}) is reduced to $h''(\theta)=0$. Note $h(0)=0$ and $h(\theta_0)=B \neq 0$, we have $\Psi(s, \theta)=h(\theta)= \frac{B}{\theta_0} \theta$. By (\ref{thm1:eq6}), we have
    \begin{equation*}
        \psi(r,\theta)=\Psi(s, \theta)=\frac{B}{\theta_0} \theta,  \quad (r,\theta)\in \overline{\Omega_{1,2,\theta_0}}.
    \end{equation*}
    Using (\ref{thm0:eq0}), we conclude that $\vu$ must be $(-1)$-homogeneous and takes the form $\vu=\frac{b}{r} \bm{e_r}$, where $b=-B/\theta_0$. Then by the first line of Euler equations (\ref{eq:euler}), we obtain
    \[
        \frac{\p P}{\p r}=\frac{b^2}{r^3}, \quad \frac{\p P}{\p \theta} =0, 
    \]
    which implies that $P=-\frac{b^2}{2r^2}+C$ for some constant $C$.  Theorem \ref{thm:1} (i) is proved.  
    
    \medskip

    \noindent 3. Next, we prove Theorem \ref{thm:1} (ii), where $c\ne 0$. We only consider the case where $c>0$ and $h$ is strictly increasing. Similar arguments, with appropriate modifications, apply to the remaining cases: $c>0$ and $h$ is strictly decreasing, $c<0$ and $h$ is strictly increasing, as well as $c<0$ and $h$ is strictly decreasing. 
    
    In this case, $c>0$ and $h$ is strictly increasing, by (\ref{thm1:eq2}), the maximum and minimum values of $\psi$ in $\overline{\Omega_{1,2,\theta_0}}$ can only be attained at $\p \Omega^{TR}_{1,2,\theta_0}$ and $\p \Omega^{BL}_{1,2,\theta_0}$, respectively. Therefore, we have $J=(0,c\ln2+B)$.
    Let $G(z):=g(z)e^{\frac{2}{c}z}$ for $z\in J$. Then $G\in C^1(J)$ since $g\in C^1(J)$. We claim that $G\equiv const$ on $J$. To see this, note by (\ref{thm1:eq11}), we have 
    \begin{equation}\label{thm1:eq17}
        G(h(\theta)+c \ln 2)=G(h(\theta)), \quad \theta \in (0,\theta_0).
    \end{equation}
    By (\ref{thm1:eq7}), we have 
    \begin{equation}\label{thm1:eq18}
        \frac{\p^2 \Psi}{\p s^2}+\frac{\p^2 \Psi}{\p \theta^2}=e^{-\frac{2}{c} \Psi} G(\Psi+cs), \quad  (s,\theta)\in  D.
    \end{equation}
    By assumption of the theorem, when $c\ne 0$, we have $\p_\theta (ru_r)|_{\theta=0}=A$ (or $\p_\theta (ru_r)|_{\theta=\theta_0}=A$). By this, (\ref{thm0:eq0}) and (\ref{thm1:eq6}), we have
    \begin{equation}\label{thm1:eq19}
        \frac{\p^2 \Psi}{\p \theta^2}(s,0)=-A \quad \bigg(\text{resp.}\ \frac{\p^2 \Psi}{\p \theta^2}(s,\theta_0)=-A \bigg),\quad s\in [0, \ln 2].
    \end{equation}
    By (\ref{thm1:eq_B}), we have 
    $\partial_s^2\Psi(s, 0)=\partial_s^2\Psi(s, \theta_0)=0$.  By this and (\ref{thm1:eq19}), taking the limit as $\theta \to 0^+$ (resp. $\theta \to \theta_0^-$) in (\ref{thm1:eq18}), we have 
    \begin{equation*}
        G(cs)=-A \quad (\text{resp.}\  G(cs+B)=-A e^{\frac{2B}{c}}),
    \end{equation*}
    for $s \in (0,\ln 2]$. 
    So $G(s)=-A$ on $(0, c\ln 2)\subset J$ (resp. $G(s) = - A \exp (2 B / c)$ on $(B, B+c\ln 2)\subset J$). By (\ref{thm1:eq17}) and Lemma \ref{lem1_2}, we have $G\equiv -A$ on $J$. 
    Then by (\ref{thm1:eq18}), we have
    \begin{equation*}
        \frac{\p^2 \Psi}{\p s^2}+\frac{\p^2 \Psi}{\p \theta^2}=-A e^{-\frac{2}{c} \Psi},  \quad (s,\theta)\in D. 
    \end{equation*}
    Applying Lemma \ref{lem:4} to $\Psi$ with $L=\p_{s}^2+\p_{\theta}^2$, $F=1$, $g(\Psi)=-Ae^{-\frac{2}{c} \Psi}$ and $\Omega=D$ there. By (\ref{thm1:eq8'})-(\ref{thm1:eq_B}) and the fact $|\p_\theta \Psi|>0$, 
    condition (\ref{eq4_1}) and  (ii) in Lemma \ref{lem:4} holds for $\Psi$, and  we have $\Psi(s,\theta)=h(\theta)$ in $\bar{D}$.
    By (\ref{thm1:eq6}), we have
    \begin{equation*}
        \psi(r,\theta)=c \ln r+h(\theta), \quad (r,\theta) \in \overline{\Omega_{1,2,\theta_0}}.
    \end{equation*}
    In view of (\ref{thm0:eq0}),  we have $\vu$ is $(-1)$-homogeneous and
    \begin{equation}{\label{thm:eq19}}
        \vu=\frac{c}{r} \bm{e_\theta}+\frac{f(\theta)}{r} \bm{e_r},
    \end{equation}
    where $f(\theta)=-h'(\theta)$ is strictly sign-definite. 
    Substituting (\ref{thm:eq19}) into the first line of (\ref{eq:euler}), we have  
    \[
        \frac{\p P}{\p r} =\frac{1}{r^3}\big( f^2(\theta)-cf'(\theta)+c^2 \big), \quad 
        \frac{\p P}{\p \theta}=0.
    \]
    This means that $f^2(\theta)-cf'(\theta)+c^2=const$, and 
    \begin{equation*}
        P=\frac{p}{r^2}+C,
    \end{equation*}
    where $C$ is an arbitrary constant, and $p$ is a constant satisfying $cf'(\theta)=f^2(\theta)+c^2+2p$.
    The proof of Theorem \ref{thm:1} is complete. 
\end{proof}

Next, we prove Theorem \ref{thm:2} with a similar approach to that for Theorem \ref{thm:1}.

\begin{proof}[Proof of Theorem \ref{thm:2}]
    \noindent 1. Let $\alpha\ne 1$. Let $J :=\{\psi(r,\theta) : (r,\theta) \in \Omega_{1,2,\theta_0} \}$. We show $\psi$ satisfies $\Delta\psi=g(\psi)$ for some $g\in C^1(J)$ and identify $\psi$ on $\partial\Omega_{1, 2, \theta_0}$.  For convenience, denote
    \begin{equation*}
        C_1:=\frac{c_1}{1-\alpha}, \quad C_2:=\frac{c_2}{1-\alpha}.
    \end{equation*}
    By (\ref{eq:BC:hom:-alpha:theta})-(\ref{thm0:eq0}) and Lemma \ref{lem1_1} with $\alpha \ne 1$, we have (\ref{eq1_1_3}). Without loss of generality, assume $A=0$ in (\ref{eq1_1_3}). We have
    \begin{equation}\label{thm2:eq2}
        \psi(r,\theta)=h(\theta) r^{1-\alpha}, \quad (r,\theta)\in \p \Omega_{1,2,\theta_0}, 
    \end{equation}
    where $h(\theta):= -\int_0^\theta u_r(1,\varphi) \text{d} \varphi+C_1\in C^3([0,\theta_0])$. 
    By assumption, we have either $u_r>0$ or $u_r<0$ in $\overline{\Omega_{1,2,\theta_0}}$. So $h=-\int_0^\theta u_r(1,\varphi) \text{d} \varphi+C_1$ is strictly monotone. So $C_1=h(0)\ne h(\theta_0)=C_2$.
    Since $\psi$ is strictly monotone with respect to $\theta$, the maximum and minimum of $\psi$ in $\overline{\Omega_{1,2,\theta_0}}$ can only be attained on $\p \Omega^T_{1,2,\theta_0} \cup \p \Omega^B_{1,2,\theta_0}$. Therefore, by the continuity of $\psi$ and the connectivity of $\overline{\Omega_{1,2,\theta_0}}$, it follows that
    \[
        J = (\min \{ C_1, C_2, 2^{1-\alpha} C_1, 2^{1-\alpha} C_2 \}, \max \{ C_1, C_2, 2^{1-\alpha} C_1, 2^{1-\alpha} C_2 \}) 
    \]
    is an open interval which does not include the extreme values of $\psi$ in $\overline{\Omega_{1,2,\theta_0}}$. 
    
    By (\ref{thm2:eq2}), we have $\psi(r, 0)=h(0)r^{1-\alpha}$ and $\psi(r, \theta_0)=h(\theta_0)r^{1-\alpha}$ are monotone for $r\in (1, 2)$. By this and the fact that $|\partial_{\theta}\psi|>0$, the condition (i) in Lemma \ref{lem:2} is satisfied. Then by Lemma \ref{lem:3}, there exists a scalar function $g \in C^1(J)$, such that
    \begin{equation}\label{thm2:eq4}
        \Delta \psi=\frac{\p^2 \psi}{\p r^2}+\frac{1}{r} \frac{\p \psi}{\p r}+\frac{1}{r^2} \frac{\p^2 \psi}{\p \theta^2}=g(\psi),  \quad (r,\theta) \in \Omega_{1,2,\theta_0}.
    \end{equation}
    By (\ref{eq:BC:hom:-alpha:r}) and (\ref{thm0:eq0}), we have
    \begin{equation}\label{thm2:eq5}
        \frac{\p \psi}{\p r}(1,\theta) =2^{\alpha}
        \frac{\p \psi}{\p r}(2,\theta),\quad
        \frac{\p^2 \psi}{\p r^2}(1,\theta)=2^{\alpha+1}
        \frac{\p^2 \psi}{\p r^2}(2,\theta), \quad \theta\in [0, \theta_0].
    \end{equation}
    Let $s:=\ln r$, $D:=(0,\ln 2) \times (0,\theta_0)$ and
    \begin{equation}\label{thm2:eq6}
        \Psi(s,\theta):=\psi(e^s,\theta)e^{s(\alpha-1)}, \quad (s,\theta) \in \overline{D}.
    \end{equation}
    Then $\Psi \in C^3(\overline{D})$ and $|\p_\theta\Psi|>0$ since $|\p_\theta\psi|>0$. By (\ref{thm2:eq2})-(\ref{thm2:eq6}), we have
    \begin{equation}\label{thm2:eq7}
        \frac{\p^2 \Psi}{\p s^2}+ \frac{\p^2 \Psi}{\p \theta^2}+ 2(1-\alpha) \frac{\p \Psi}{\p s}+(1-\alpha)^2 \Psi =e^{(1+\alpha)s} g(e^{(1-\alpha)s} \Psi), \quad (s,\theta)\in D,
    \end{equation}
    and  
    \begin{equation}\label{thm2:eq8} 
    \begin{split}
        \Psi(s,\theta)&=h(\theta), \quad \text{ on } \p D, \\
        \frac{\p \Psi}{\p s}(0,\theta)= \frac{\p \Psi}{\p s}(\ln 2,\theta),\quad
        \frac{\p^2 \Psi}{\p s^2}&(0,\theta)= \frac{\p^2 \Psi}{\p s^2}(\ln 2,\theta), \quad \theta \in [0,\theta_0].
    \end{split}
    \end{equation}
    In particular, we have, for $s\in [0, \ln 2]$, that 
    \begin{equation}\label{thm2:eqB}
        \Psi(s, 0)=h(0)=C_1, \quad \Psi(s, \theta_0)=h(\theta_0)=C_2, \quad \partial_s^j\Psi(s, 0)=\partial_s^j\Psi(s, \theta_0)=0, \quad j=1, 2. 
    \end{equation}
    Let $s \to 0^+$ and $s \to (\ln 2)^-$ in (\ref{thm2:eq7}), by (\ref{thm2:eq8}), we obtain 
    \begin{equation*}
        g(h(\theta))=2^{1+\alpha} g(2^{1-\alpha} h(\theta)), \quad  \theta \in (0,\theta_0),
    \end{equation*}
    namely,
    \begin{equation}\label{thm2:eq11}
        g(s) = 2^{1+\alpha} g(2^{1-\alpha} s), \quad s\in (\min\{C_1,C_2\}, \max\{C_1,C_2\}).
    \end{equation}

    Let $J^+ := J\cap \{z>0\}$, $J^- := J\cap\{z<0\}$, and $J^0 := J\cap\{0\}$. Then $J = J^+ \cup J^- \cup J^0$. Let $I^+ := (0, \theta_0)\cap\{\theta: h(\theta)>0\}$ and $I^- := (0, \theta_0)\cap\{\theta: h(\theta)<0\}$. By (\ref{thm2:eq11}), we have if $J^0 \not= \emptyset$, then $g(0)=0$. 

    \medskip

    \noindent 3. Next, we prove that $\Psi(s,\theta)=h(\theta)$ in $\overline{D}$. 

    \noindent\textbf{Claim 1.} Assume $\p_\theta (r^{\alpha} u_r)|_{\theta=0}=c_3$. If $C_1>0$, then $g(z)=B_1^{+}|z|^{\frac{\alpha+1}{\alpha-1}}$ in $J^+$ for some constant $B_1^+$. If $C_1<0$, then $g(z)=B_1^{-}|z|^{\frac{\alpha+1}{\alpha-1}}$ in $J^-$ for some constant $B_1^-$.

    Indeed, if $\p_\theta (r^{\alpha} u_r)|_{\theta=0}=c_3$, then by (\ref{thm0:eq0}) and (\ref{thm2:eq6}), we have
    \begin{equation*}
        \frac{\p^2 \Psi}{\p \theta^2}(s,0)=-c_3, \quad s \in [0,\ln 2].
    \end{equation*}
    Sending $\theta\to 0^+$ in (\ref{thm2:eq7}), by (\ref{thm2:eqB}) and the above, we have 
    \begin{equation}\label{eq2_new_1}
        e^{(1+\alpha)s} g(e^{(1-\alpha)s}C_1)=(1-\alpha)^2C_1-c_3, \quad s\in (0, \ln 2).
    \end{equation}
    Thus for $z\in (\min\{C_1, 2^{1-\alpha} C_1\}, \max\{C_1, 2^{1-\alpha} C_1\})$, we have 
    \[
        g(z) = C |z|^{\frac{\alpha+1}{\alpha-1}}. 
    \]
    By (\ref{thm2:eq11}), (\ref{eq2_new_1}) and Lemma \ref{lem1_3}, we know that Claim 1 is true.

    \medskip

    Similarly, we have

    \noindent\textbf{Claim 2.} Assume $\p_\theta (r^\alpha u_r)|_{\theta=\theta_0}=c_4$ for some constant $c_4$. If $C_2>0$, then $g(z)=B_2^{+}|z|^{\frac{\alpha+1}{\alpha-1}}$ in $J^+$ for some constant $B_2^+$. If $C_2<0$, then $g(z)=B_2^{-}|z|^{\frac{\alpha+1}{\alpha-1}}$ in $J^-$ for some constant $B_2^-$.

    \noindent\textbf{Claim 3.} If one of (A1)-(A4) holds and $J^+$ or $J^- \not=\emptyset$, then 
    \begin{equation}\label{eq2_new_4}
        g(z)=C^{\pm}|z|^{\frac{\alpha+1}{\alpha-1}}, \quad z\in J^{\pm}
    \end{equation}
    for some constants $C^{\pm}$ respectively. 
    
    \noindent \emph{Proof of Claim 3}. Note we have $C_1\ne C_2$. Without loss of generality, assume $C_2>C_1$. 
      
    If (A1) holds, we have $C_2>C_1>0$, and either $\p_\theta (r^{\alpha} u_r)|_{\theta=0}=c_3$ or  $\p_\theta (r^\alpha u_r)|_{\theta=\theta_0}=c_4$ for some constants $c_3, c_4$.  Note the range of $h$ is $(C_1, C_2)$. By (\ref{thm2:eq2}), we have $J=(C_1, 2^{1-\alpha}C_2)$ if $\alpha<1$, and $(2^{1-\alpha}C_1, C_2)$ if $\alpha>1$. In both cases $J\subset (0, \infty)$. Then by Claim 1 and 2, we have (\ref{eq2_new_4}) for $z\in J^+=J$.

    If (A2) holds, we have $C_2>C_1=0$ and   $\p_\theta (r^\alpha u_r)|_{\theta=\theta_0}=c_4$. The range of $h$ is $(0, C_2)$. By (\ref{thm2:eq2}), we have $J=(0, 2^{1-\alpha}C_2)$ if $\alpha<1$, and $(0, C_2)$ if $\alpha>1$. In both cases $J\subset (0, \infty)$. Then by Claim 1 and 2, we have (\ref{eq2_new_4}) for $z\in J^+=J$. 
 
    If (A3) holds, we have $0=C_2>C_1$ and   $\p_\theta (r^\alpha u_r)|_{\theta=0}=c_3$. The range of $h$ is $(C_1, 0)$. By (\ref{thm2:eq2}), we have $J=(2^{1-\alpha}C_1, 0)$ if $\alpha<1$, and $(C_1, 0)$ if $\alpha>1$. In both cases $J\subset (-\infty, 0)$. Then by Claim 1 and 2, we have (\ref{eq2_new_4}) for $z\in J^-=J$. 
  
    If (A4) holds, we have $C_1<0<C_2$, and both $\p_\theta (r^\alpha u_r)|_{\theta=0}=c_3$ and  $\p_\theta (r^\alpha u_r)|_{\theta=\theta_0}=c_4$. By Claim 1, (\ref{eq2_new_4}) holds for $z\in J^-$. By Claim 2,  (\ref{eq2_new_4}) holds for $z\in J^+$. Claim 3 is proved.
    
    So 
    \begin{equation}\label{thm2:eq:200}
        g(z)=
        \begin{cases}
        C^+|z|^{\frac{\alpha+1}{\alpha-1}}, &z\in J^+, \text{if }J^+ \not= \emptyset,  \\
        0, & z\in J^0, \text{if }J^0 \not= \emptyset, \\
        C^-|z|^{\frac{\alpha+1}{\alpha-1}}, &z\in J^-, \text{if }J^- \not= \emptyset.
        \end{cases}
    \end{equation}

    Under condition (A1), by the above arguments, we have $J\subset\subset (0, \infty)$ or $(-\infty, 0)$ and $g$ is Lipschitz continuous for all $z\in \bar{J}$. Under condition (A2), (A3) or (A4), we have $\frac{\alpha+1}{\alpha-1}>1$, and $g$ is Lipschitz for all $z\in \bar{J}$. 
    Substituting (\ref{thm2:eq:200}) into (\ref{thm2:eq7}), we have 
    \begin{equation*}
        \frac{\p^2 \Psi}{\p s^2}+ \frac{\p^2 \Psi}{\p \theta^2}+ 2(1-\alpha) \frac{\p \Psi}{\p s}+(1-\alpha)^2 \Psi =g(\Psi) \text{ in } D.
    \end{equation*}
    Applying Lemma \ref{lem:4} with $L=\p_{s}^2+\p_{\theta}^2+2(1-\alpha)\p_s+(1-\alpha)^2$, $F=1$, and $g$ defined in (\ref{thm2:eq:200}), and utilizing the boundary conditions (\ref{thm2:eq8}), (\ref{thm2:eqB}) along with $|\p_\theta \Psi|>0$, we conclude that 
    \begin{equation}\label{eq2_new_5}
        \Psi(s,\theta)=h(\theta)  \text{ in } \overline{D},
    \end{equation}
    where $h$ is a strictly monotone function satisfying $h(0) = C_1$ and $h(\theta_0)=C_2$.
    
    \medskip

    \noindent 3.  
    By (\ref{thm2:eq6}) and (\ref{eq2_new_5}), we have
    \begin{equation*}
        \psi(r,\theta)=h(\theta) r^{1-\alpha}, \quad   \text{ in } \overline{\Omega_{1,2,\theta_0}}.
    \end{equation*}
    By (\ref{thm0:eq0}), 
    we conclude that $\vu$ must be $(-\alpha)$-homogeneous and takes the form
    \begin{equation}{\label{e31}}
        \vu=\frac{v(\theta)}{r^\alpha} \bm{e_\theta}+\frac{f(\theta)}{r^\alpha} \bm{e_r},
    \end{equation}
    where $f(\theta)=-h'(\theta)$ is strictly sign-definite for $\theta \in [0,\theta_0]$ in cases (A1)-(A4), $v(\theta)=(1-\alpha) h(\theta)$ is strictly monotone. Clearly,
    \begin{equation}\label{thm2:eq51}
        (1-\alpha)f+v'=0.
    \end{equation}
    Note that the function $v$ satisfies 
    $v(0)=(1-\alpha)C_1=c_1$ and $v(\theta_0)=(1-\alpha)C_2=c_2$. Therefore, $|v|>0$ for $\theta \in (0,\theta_0)$ when $c_1 c_2 \geqslant 0$, and $v$ is sign-changing when $c_1 c_2<0$. 
    
    Substituting (\ref{e31}) into the first line of (\ref{eq:euler}), we have
    \[
        \frac{\p P}{\p r} =\frac{1}{r^{2\alpha+1}} ( \alpha f^2-vf'+v^2 ), \quad \frac{\p P}{\p \theta} =\frac{1}{r^{2\alpha}} \big( (\alpha-1)vf-vv' \big). 
    \]
    By (\ref{thm2:eq51}) and the second identify in the above, we have $\partial_{\theta}P=0$. Then taking $\partial_{\theta}$ of the first identify, we have $\partial_{\theta}\partial_rP=0$ and consequently $\alpha f^2-vf'+v^2=const$. Then the first identity is reduced to $\partial_rP=c_p/r^{2\alpha+1}$ for some constant $c_p$, and thus
    \begin{equation*}
        P=\frac{p}{r^{2\alpha}}+C,
    \end{equation*}
    where $C$ is an arbitrary constant, and $p$ is a constant satisfying $vf'-\alpha f^2-v^2=2\alpha p$.
    This completes the proof of Theorem \ref{thm:2}. 
\end{proof}

We now proceed to prove Theorem \ref{thm:3}. 

\noindent
\emph{Proof of Theorem \ref{thm:3}}:

    \noindent 1. Let $J:=\{\psi(r,\theta):(r,\theta) \in \Omega_{1,2,\theta_0} \}$. We show $\psi$ satisfies $\Delta\psi=g(\psi)$ for some $g\in C^1(J)$ and identify $\psi$ on $\partial\Omega_{1, 2, \theta_0}$. 

    By assumption, $u_r|_{r=1}=u_r|_{r=2}=0$. By this, (\ref{eq:BC:hom:-alpha:theta}), (\ref{thm0:eq0}) and Lemma \ref{lem1_1}, we have (\ref{eq1_1_2}) and (\ref{eq1_1_3}) with $u_r(1, \theta)=0$. Without loss of generality, let $A=0$ in  (\ref{eq1_1_2}) and (\ref{eq1_1_3}). Then
    \begin{equation}\label{thm3:eq2}
        \psi(r,\theta) = 
        \left\{
        \begin{array}{ll}
            c \ln r, & \textrm{ if }\alpha=1,\\
            \displaystyle \frac{c}{1-\alpha} r^{1-\alpha}, & \textrm{ if }\alpha\ne 1,
        \end{array}
        \right. 
        \quad (r,\theta)\in \p \Omega_{1,2,\theta_0}, 
    \end{equation}
    where $c_1=c_2=c\not=0$. 

    In particular, if $\alpha=1$, then  $J=(\min \{0, c \ln 2 \}, \max \{0, c \ln 2 \})$. If $\alpha>1$, then  $J=(\min \{\frac{c}{1-\alpha}, \frac{c}{1-\alpha}2^{1-\alpha} \}, \max \{\frac{c}{1-\alpha}, \frac{c}{1-\alpha}2^{1-\alpha} \})$. 
    In either case, condition (ii) in Lemma \ref{lem:2} holds, and Lemma \ref{lem:3} implies that there exists a scalar function $g \in C^1(J)$, such that 
    \begin{equation}\label{thm3:eq4}
        \Delta \psi=\frac{\p^2 \psi}{\p r^2}+\frac{1}{r} \frac{\p \psi}{\p r}+\frac{1}{r^2} \frac{\p^2 \psi}{\p \theta^2}=g(\psi), \quad (r,\theta)\in \Omega_{1,2,\theta_0}.
    \end{equation}

    By (\ref{thm0:eq0}) and boundary condition $u_r|_{\theta=0}=u_r|_{\theta=\theta_0}$, it follows that for all $r \in [1,2]$, 
    \begin{equation}\label{thm3:eq5}
        \frac{\p \psi}{\p \theta}(r,0)=
        \frac{\p \psi}{\p \theta}(r,\theta_0).
    \end{equation}
    By continuity of $\psi$, (\ref{thm3:eq2}) and (\ref{thm3:eq4}), it is easy to check that for all $r \in [1,2]$, 
    \begin{equation}\label{thm3:eq55}
        \frac{\p^2 \psi}{\p \theta^2}(r,0)=
        \frac{\p^2 \psi}{\p \theta^2}(r,\theta_0).
    \end{equation}
    Let $s=\ln r$, $D:=(0, \ln 2) \times (0,\theta_0)$ and
    \begin{equation}\label{thm3:eq6}
        \Psi(s,\theta):=\psi(e^s,\theta), \quad (s,\theta)\in \overline{D}.
    \end{equation}
    Then $\Psi \in C^3(\overline{D})$ and $|\p_s\Psi|=e^s|\p_r\psi|>0$. By (\ref{thm3:eq4})-(\ref{thm3:eq6}), we have
    \begin{equation}\label{thm3:eq7}
        \frac{\p^2 \Psi}{\p s^2}+\frac{\p^2 \Psi}{\p \theta^2}=e^{2s} g(\Psi), \quad (s,\theta)\in D,
    \end{equation}
    and 
    \begin{equation}\label{thm3:eq8} 
        \frac{\p \Psi}{\p \theta}(s,0)= \frac{\p \Psi}{\p \theta}(s,\theta_0),\quad
        \frac{\p^2 \Psi}{\p \theta^2}(s,0)= \frac{\p^2 \Psi}{\p \theta^2}(s,\theta_0), \quad s \in [0,\ln 2].
    \end{equation}

    \noindent 2. Next, we show $\vu$ must be $(-\alpha)$-homogeneous and takes the form
    \begin{equation}\label{thm3:eq1015}
        \vu=\frac{c}{r^\alpha} \bm{e_\theta}.
    \end{equation}

    We first prove this when $\alpha=1$. In this case,  (\ref{thm3:eq2}) and (\ref{thm3:eq6}) imply
    \begin{equation}\label{thm3:eq9} 
        \Psi(s,\theta)=cs,\quad (s,\theta)\in \p D,
    \end{equation}
    Taking the limits as $\theta \to 0^+$ in (\ref{thm3:eq7}), note the above implies $\partial_s^2\Psi(s, 0)=0$, we have
    \begin{equation}\label{thm3:eq11}
        g(cs)=e^{-2s}\frac{\p^2 \Psi}{\p \theta^2}(s,0), \quad s\in (0, \ln 2). 
    \end{equation}
    Since $\Psi \in C^3(\overline{D})$ and $g \in C^1(J)$,  differentiate both sides of (\ref{thm3:eq11}) with respect to $s$, we know that $g' \in L^\infty(J)$, therefore $g$ is a Lipschitz function. 
    
    If $c>0$, then the assumptions in the theorem imply that $\p_\theta (r u_r)|_{\theta=0} \geqslant 0$ or $\p_\theta (r u_r)|_{\theta=\theta_0} \geqslant 0$. By (\ref{thm0:eq0}), (\ref{thm3:eq6}) and (\ref{thm3:eq8}), we have
    \begin{equation*}
        \frac{\p^2 \Psi}{\p \theta^2}(s,0) \leqslant 0, \quad s\in [0,\ln 2].
    \end{equation*}
    Combining this with (\ref{thm3:eq11}), we conclude that $g \leqslant 0$ in $J$. 
    
    Applying Lemma \ref{lem:4} with $L=\p_{s}^2+\p_{\theta}^2$, $F(s)=e^{2s}$, $g$ defined in (\ref{thm3:eq11}) and $\Omega=D$ there.  Using the boundary conditions (\ref{thm3:eq8}) and (\ref{thm3:eq9}), along with $|\p_s \Psi|>0$, we conclude that condition (ii) in Lemma \ref{lem:4} holds, and
    \begin{equation}\label{thm3:eq13}
        \Psi(s,\theta)=cs, \quad (s,\theta)\in \overline{D}.
    \end{equation}

    If $c<0$, then the assumptions in the theorem imply that $\p_\theta (r u_r)|_{\theta=0} \leqslant 0$ or $\p_\theta (r u_r)|_{\theta=\theta_0} \leqslant 0$. Then a similar argument shows that $g \geqslant 0$. 
    Lemma \ref{lem:4} implies that (\ref{thm3:eq13}) remains valid and thus
    \begin{equation*}
        \psi(r,\theta)=c \ln r, \quad (r,\theta)\in \overline{\Omega_{1,2,\theta_0}}.
    \end{equation*}
    Then from (\ref{thm0:eq0}), it holds that $\vu$ must be $(-1)$-homogeneous, taking the form (\ref{thm3:eq1015}). 
    
    \noindent 3. Next, we prove the theorem  when $\alpha>1$. By (\ref{thm3:eq2}) and (\ref{thm3:eq6}), we have
    \begin{equation}\label{thm3:eq109} 
        \Psi(s,\theta)=\frac{c}{1-\alpha}e^{s(1-\alpha)}, \quad (s,\theta)\in \p D.
    \end{equation}
    Taking the limits as $\theta \to 0^+$ in (\ref{thm3:eq7}), we find that for all $s \in (0,\ln 2)$, 
    \begin{equation}\label{thm3:eq1011}
        g \Big(\frac{c}{1-\alpha}e^{s(1-\alpha)} \Big)=e^{-2s} \bigg\{c(1-\alpha)e^{s(1-\alpha)}+ \frac{\p^2 \Psi}{\p \theta^2}(s,0) \bigg\} .
    \end{equation}
    Since $\Psi \in C^3(\overline{D})$, $g \in C^1(J)$ and $0 \notin \bar{J}$, (\ref{thm3:eq1011}) implies that $g$ is a Lipschitz function.  
    
    If $c>0$, the assumptions in Theorem \ref{thm:3} imply that $\p_\theta(r^\alpha u_r)|_{\theta=0} \geqslant c(1-\alpha)$ or $\p_\theta (r^\alpha u_r)|_{\theta=\theta_0} \geqslant c(1-\alpha)$. By (\ref{thm0:eq0}), (\ref{thm3:eq6}) and (\ref{thm3:eq8}), we have that
    \begin{equation*}
        \frac{\p^2 \Psi}{\p \theta^2}(s,0) \leqslant c(\alpha-1)e^{s(1-\alpha)}, \quad s\in [0,\ln 2].
    \end{equation*}
    Combining this with (\ref{thm3:eq1011}), we know that $g \leqslant 0$ in $J$. 
    Applying Lemma \ref{lem:4} with $L=\p_{s}^2+\p_{\theta}^2$, $F(s)=e^{2s}$, $g$ defined in (\ref{thm3:eq1011}), and $\Omega=D$ there, using the boundary conditions (\ref{thm3:eq8}) and (\ref{thm3:eq109}) along with $|\p_s\Psi|>0$, we conclude that
    \begin{equation}\label{thm3:eq1013}
        \Psi(s,\theta)=\frac{c}{1-\alpha}e^{s(1-\alpha)}, \quad (s,\theta)\in \overline{D}.
    \end{equation}

    If $c<0$, then $\p_\theta(r^\alpha u_r)|_{\theta=0} \leqslant c(1-\alpha)$ or $\p_\theta (r^\alpha u_r)|_{\theta=\theta_0} \leqslant c(1-\alpha)$. 
    A similar argument as above shows that $g \geqslant 0$, and Lemma \ref{lem:4} implies that (\ref{thm3:eq1013}) holds.
    
    We obtain from (\ref{thm3:eq6}) that
    \begin{equation*}
        \psi(r,\theta)=\frac{c}{1-\alpha}r^{1-\alpha},\quad (r,\theta)\in \overline{\Omega_{1,2,\theta_0}}.
    \end{equation*}
    Then from (\ref{thm0:eq0}), it holds that $\vu$ must be $(-\alpha)$-homogeneous, taking the form (\ref{thm3:eq1015}). 
    So (\ref{thm3:eq1015}) holds for all $\alpha \geqslant 1$.
    Substituting (\ref{thm3:eq1015}) into the first line of (\ref{eq:euler}) leads to
    \[
        \frac{\p P}{\p r} =\frac{c^2}{r^{2\alpha+1}}, \quad \frac{\p P}{\p \theta} =0.
    \]
    It implies that 
    \begin{equation*}
        P=-\frac{c^2}{2\alpha r^{2\alpha}}+C,
    \end{equation*}
    where $C$ is an arbitrary constant. 
    This completes the proof of Theorem \ref{thm:3}. 
\qed

\begin{remark}
    If $\alpha<1$ in Theorem \ref{thm:3}, in order to prove the rigidity of $(-\alpha)$-homogeneous solutions to the Euler equations (\ref{eq:euler}), we expect the velocity field $\vu$ to take the form $\vu = \frac{c}{r^\alpha} \bm{e_\theta}$ with $c \neq 0$, which is equivalent to the stream function having the form (\ref{thm3:eq2}). However, direct calculation shows that the function $g$ in equation (\ref{thm3:eq7}) does not satisfy the sign conditions (ii) or (iii) of Lemma \ref{lem:4}. Therefore, our method is not suitable for the case $\alpha<1$.
\end{remark}

\bigskip

Next, we prove Theorem $\ref{thm:4}$. 
    
\noindent
\emph{Proof of Theorem \ref{thm:4}}: 
    
    \noindent 1. 
    Note that $(a, b)=(0, 1)$ or $(1, \infty)$. By assumption, we have $u_\theta|_{\theta=0}=u_\theta|_{\theta=\theta_0}=u_\theta|_{r=1} = c/r$ for some constant $c$. Note $\partial_{r}\psi=u_{\theta}$ and $\psi\in C^3(\overline{\Omega_{a, b, \theta_0}}\setminus \{ \bm{0}\})$, we have 
    \begin{equation}\label{thm4:eq1}
        \psi(r,\theta)=c \ln r+h(\theta),  \text{ in } \p \Omega_{a, b,\theta_0}\setminus \{ \bm{0}\}, 
    \end{equation}
    for some $h(\theta)\in C^3([0,\theta_0])$, and 
    \begin{equation}\label{thm4:eq3}
        \frac{\p \psi}{\p r}(1,\theta)=c, \quad \theta \in [0,\theta_0]. 
    \end{equation}

    Without loss of generality, assume  $h(0)=0$. 
    Note by (\ref{thm0:eq0}),  $\partial_{\theta}\psi=-ru_r$  in $\overline{\Omega_{a, b, \theta_0}}\setminus \{ \bm{0}\}$.  If $u_r>0$ in $\overline{\Omega_{a, b, \theta_0}}\setminus \{ \bm{0}\}$, then $\p_\theta \psi<0$ in $\overline{\Omega_{a, b, \theta_0}}\setminus \{ \bm{0}\}$, and $h(\theta)$ is strictly decreasing in $[0,\theta_0]$. 
    Similarly, if $u_r<0$ in $\overline{\Omega_{a, b, \theta_0}}\setminus \{ \bm{0}\}$, then $h(\theta)$ is strictly increasing in $[0,\theta_0]$.

    Define $J:=\{\psi(r,\theta):(r,\theta) \in \Omega_{a, b,\theta_0} \}$. The assumption that $|\partial_{\theta}\psi|>0$ in $\overline{\Omega_{a, b, \theta_0}}\setminus \{ \bm{0}\}$ implies that $\psi$ attains its extreme values on $\partial \Omega_{a, b,\theta_0}\setminus\{\bm{0}\}$.
    If $c=0$, it is clear that $J=(0,h(\theta_0))$ when $u_r<0$, and $J=(h(\theta_0),0)$ when $u_r>0$. 

    For both $(a, b)=(0, 1)$ and $(1, \infty)$, if $c \vu \cdot \bm{\nu} |_{r=1} > 0$, it can be checked that $\psi$ attains its extreme value in $\overline{\Omega_{a, b, \theta_0}}\setminus \{ \bm{0}\}$ only at $(r, \theta)=(1, 0)$, and $J=\psi(\p \Omega_{a, b,\theta_0}^{B} \setminus \{\bm{0}, (1, 0)\}$ is an open unbounded interval. 
    Similarly, if $c \vu\cdot \bm{\nu} |_{r=1} < 0$, it can be checked that $\psi$ attains its extreme value in $\overline{\Omega_{a, b, \theta_0}}\setminus \{ \bm{0}\}$ only at $(1, \theta_0)$, and $J=\psi(\p \Omega_{1,\infty,\theta_0}^{T} \setminus \{\bm{0}, (1, \theta_0)\}$ is an open unbounded interval.
    
    By (\ref{thm4:eq1}) and the fact $|\partial_{\theta}\psi|>0$ in $\overline{\Omega_{a, b, \theta_0}}\setminus \{ \bm{0}\}$, we also know condition (i) in Lemma \ref{lem:2} holds. By Lemma \ref{lem:3}, there exists a function $g \in C^1(J)$ such that 
    \begin{equation}\label{thm4:eq2}
        \Delta \psi=\frac{\p^2 \psi}{\p r^2}+\frac{1}{r} \frac{\p \psi}{\p r}+\frac{1}{r^2} \frac{\p^2 \psi}{\p \theta^2}=g(\psi) \text{ in } \Omega_{a, b,\theta_0}.
    \end{equation}
    Next, let $s:=\ln r$,  $D:=(\ln a, \ln b) \times (0,\theta_0)$ and 
    \begin{equation}\label{thm4:eq4}
        \Psi(s,\theta):=\psi(e^s,\theta)-cs, \quad (s,\theta) \in \overline{D}. 
    \end{equation}
    Them $\Psi \in C^3(\overline{D})$ and $|\p_\theta \Psi|=|\p_\theta \psi|>0$. By (\ref{thm4:eq1})-(\ref{thm4:eq4}), we have
    \begin{equation}\label{thm4:eq5}
        \frac{\p^2 \Psi}{\p s^2}+\frac{\p^2 \Psi}{\p \theta^2}=e^{2s} g(\Psi+cs), \quad (s,\theta)\in D,
    \end{equation}
    and 
    \begin{equation}\label{thm4:eq6}
        \begin{split}
        &\Psi(s,\theta) =h(\theta), \quad \text{ on } \p D, \\
        &\p_s \Psi(0,\theta) =0, \quad \theta \in [0,\theta_0]. 
        \end{split}
    \end{equation}

    \noindent 2. We first prove the theorem if (B1) or (B2) holds. If (B1) holds, we have $c=0$, $u_r<0$, and $\p_\theta u_r|_{r=1}\geqslant \p_r (ru_\theta)|_{r=1}$ (resp. if (B2) holds, then $c=0$, $u_r>0$ and $\p_\theta u_r|_{r=1}\leqslant \p_r (ru_\theta)|_{r=1}$). Then by this, (\ref{thm0:eq0}), and (\ref{thm4:eq4}), we have
    \begin{equation}\label{thm4:eq4000}
        \frac{\p^2 \Psi}{\p s^2}+\frac{\p^2 \Psi}{\p \theta^2}(0,\theta)\leqslant 0, \quad (\text{resp. } \frac{\p^2 \Psi}{\p s^2}+\frac{\p^2 \Psi}{\p \theta^2}(0,\theta) \geqslant 0), \quad \theta \in [0,\theta_0]. 
    \end{equation}
    Taking the limit $s \to 0$ in (\ref{thm4:eq5}), we have  
    \begin{equation}\label{thm4:eq4001}
        g(h(\theta))=\frac{\p^2 \Psi}{\p s^2}(0,\theta)+\frac{\p^2 \Psi}{\p \theta^2} (0,\theta), \quad \theta \in (0,\theta_0). 
    \end{equation}
    Differentiating (\ref{thm4:eq4001}) with respect to $\theta$. Note $h'(\theta)$ having a strict sign in $[0,\theta_0]$ and $\Psi \in C^3(\overline{D})$, we conclude that $g'(z)$ is bounded in $J$, and hence $g$ is Lipschitz continuous in $J$. Moreover, (\ref{thm4:eq4000}) implies that $g \leqslant 0$ (resp. $g \geqslant 0$).

    Applying Lemma \ref{lem:4} with $L=\p_{s}^2+\p_{\theta}^2$, $F(s)=e^{2s}$, $g$ defined in (\ref{thm4:eq4001}), and $\Omega=D$ there. Note we have $D=(-\infty, 0)\times (0, \theta_0)$ or $(0, \infty)\times (0, \theta_0)$. Using the boundary conditions (\ref{thm4:eq6}) along with $|\p_\theta \Psi|>0$ in $\Omega_{a, b, \theta_0}$, noting that either (iii) or (iv) in Lemma \ref{lem:4} holds, we conclude that $\Psi(s,\theta)=h(\theta)$ in $\bar{D}$. By (\ref{thm4:eq4}), we have 
    $\psi(r,\theta)=h(\theta)$. In view of (\ref{thm0:eq0}), we have
    \begin{equation}\label{thm4:eq6000}
        \vu=\frac{f(\theta)}{r} \bm{e_r}, \quad (r,\theta) \in \overline{\Omega_{a, b, \theta_0}}\setminus \{ \bm{0}\},
    \end{equation}
    where $f(\theta)=-h'(\theta)$ is strictly sign-definite. Substituting (\ref{thm4:eq6000}) into the first line of Euler equations (\ref{eq:euler}) gives 
    \begin{equation*}
        \frac{\p P}{\p r}=\frac{f^2(\theta)}{r^3},\quad \frac{\p P}{\p \theta}=0.
    \end{equation*}
    This implies that $f$ must be a constant, and
    \begin{equation*}
        P=\frac{p}{r^2}+C, 
    \end{equation*}
    where $p=-f^2/2$, and $C$ is an arbitrary constant. So the conclusion of Theorem \ref{thm:4} holds under condition (B1) or (B2).
    
    \medskip
    
    \noindent 3. Now we prove the theorem if (B3) or (B4) holds.

    \noindent \textbf{Claim}: If (B3) or (B4) holds,  then $g(z)=C e^{-\frac{2}{c}z},\ z \in J$ for some constant $C$.

    \noindent \textit{Proof of Claim:} 
    If (B3) holds, we have $\p_\theta(ru_r)|_{\theta=0}=A$. In view of (\ref{thm0:eq0}),  it implies
    \begin{equation}\label{thm4:eq7}
        \frac{\p^2 \Psi}{\p \theta^2}(s,0)=-A, 
    \end{equation}
    for $s\in [0,+\infty)$ if $a=1, b=\infty$, and $s\in (-\infty, 0]$ if $a=0, b=1$.
    Taking the limit as $\theta \to 0^+$ in (\ref{thm4:eq5}), by the first line of (\ref{thm4:eq6}), (\ref{thm4:eq7}) and $J=\psi(\p \Omega_{a, b, \theta_0}^{B} \setminus \{\bm{0}, (1, 0)\})$, we find that $g(z)=-A e^{-\frac{2}{c}z},\ z \in J$.

    If (B4) holds, we have $c \vu \cdot \bm{\nu} |_{r=1} < 0$ in $\overline{\Omega_{a, b, \theta_0}}\setminus \{ \bm{0}\}$ and $\p_\theta(ru_r)|_{\theta=\theta_0}=A$. In view of (\ref{thm0:eq0}), it follows that
    \begin{equation}\label{thm4:eq8}
        \frac{\p^2 \Psi}{\p \theta^2}(s,\theta_0)=-A. 
    \end{equation}
    Note that $J=\psi(\p \Omega_{a, b, \theta_0}^{T} \setminus \{\bm{0}, (1, \theta_0)\})$. Taking the limit as $\theta \to \theta_0^-$ in (\ref{thm4:eq5}), by (\ref{thm4:eq6}), (\ref{thm4:eq8}), there exists a constant $C$ such that $g(z)=C e^{-\frac{2}{c}z},\ z \in J$. The claim is proved. 

    Substituting $g(z)=C e^{-\frac{2}{c}z}$ and $z \in J$ into (\ref{thm4:eq5}), we have
    \begin{equation*}
        \frac{\p^2 \Psi}{\p s^2}+\frac{\p^2 \Psi}{\p \theta^2}=Ce^{-\frac{2}{c} \Psi}, \quad (s,\theta)\in D.
    \end{equation*} 
    We apply Lemma \ref{lem:4} to $\Psi$ with $L=\p_{s}^2+\p_{\theta}^2$, $F=1$, $g(\Psi)=Ce^{-\frac{2}{c} \Psi}$ and $\Omega=D$ there. By (\ref{thm4:eq6}) and the fact $|\p_\theta \Psi|>0$ in $\Omega_{a, b, \theta_0}$, we conclude that $\Psi(s,\theta)=h(\theta)$ in $\bar{D}$. In view of (\ref{thm4:eq4}), we have $\psi(r,\theta)=c\ln r+h(\theta)  \text{ in } \overline{\Omega_{a, b, \theta_0}}\setminus \{ \bm{0}\}$. 
    Then the same argument as in the proof of Theorem \ref{thm:1} (ii) yields the assertion of Theorem \ref{thm:4}.
\qed

\bigskip

Finally, we present the proof of Theorem \ref{thm:5}.

\noindent
\emph{Proof of Theorem \ref{thm:5}}:
    We prove this theorem for  $\alpha = 1$ and $\alpha \neq 1$ respectively.

    \noindent\textbf{(i)} $\alpha=1$. 
    By (\ref{eq:BC:hom:-alpha:theta}) with $\alpha=1$, we have $u_\theta|_{\theta=0}=c_1/r$, $u_\theta|_{\theta=\theta_0}=c_2/r$. By this and  (\ref{thm0:eq0}), we have 
    \[
        \psi(r,0)=c_1 \ln r+A_1, \quad \psi(r,\theta_0)=c_2 \ln r+A_2, \quad 0<r<\infty,
    \]
    for some constants $A_1, A_2$.  We claim that $c_1=c_2$. If not, then by the above, there exists some $r_*\in (0, \infty)$, such that $\psi(r_*,0)=\psi(r_*, \theta_0)$. Note by assumption, we have $|\partial_{\theta}\psi|=|ru_r|>0$. So $\psi(r_*, \theta)$ is strictly monotone in $\theta\in [0, \theta_0]$. Thus $\psi(r_*,0)\ne \psi(r_*, \theta_0)$ and we have a contradiction. So $c_1=c_2=c$ for some constant $c$, and 
    \begin{equation}\label{thm6:eq1}
        \psi(r,0)=c \ln r+A_1, \quad \psi(r,\theta_0)=c \ln r+A_2, \quad 0<r<\infty. 
    \end{equation}

    By this and the assumption that $|\partial_{\theta}\psi|>0$,  condition (i) in Lemma \ref{lem:2} is satisfied. By Lemma \ref{lem:3}, there exists a function $g \in C^1(\mathbb{R})$ such that 

    \begin{equation}\label{thm6:eq2}
        \Delta \psi=\frac{\p^2 \psi}{\p r^2}+\frac{1}{r} \frac{\p \psi}{\p r}+\frac{1}{r^2} \frac{\p^2 \psi}{\p \theta^2}=g(\psi) \text{ in } \Omega_{0,\infty,\theta_0}. 
    \end{equation}
    Let $s=\ln r$, $D:=\mathbb{R} \times (0,\theta_0)$ and  
    \begin{equation}\label{thm6:eq3}
        \Psi(s,\theta):=\psi(e^s,\theta)-cs, \quad (s,\theta) \in \overline{D}. 
    \end{equation}
    Then $\Psi \in C^3(\overline{D})$  and $|\p_\theta \Psi|=|\p_\theta \psi|>0$. 
    By (\ref{thm6:eq1})-(\ref{thm6:eq3}), we have 
    \begin{equation}\label{thm6:eq4}
        \frac{\p^2 \Psi}{\p s^2}+\frac{\p^2 \Psi}{\p \theta^2}=e^{2s} g(\Psi+cs), \quad (s,\theta)\in D,
    \end{equation}
    and boundary conditions 
    \begin{equation}\label{thm6:eq5}
        \Psi(s,0)=A_1, \quad \Psi(s,\theta_0)=A_2 \neq A_1. 
    \end{equation}
    
    Define $J:=\{\psi(r,\theta):(r,\theta)\in \Omega_{0,\infty,\theta_0} \}$. 
    Note $\partial_{\theta}\psi=-ru_r$. If $c=0$, it is clear that $J=(A_1, A_2)$ when $u_r<0$, and $J=(A_2, A_1)$ when $u_r>0$.  If $c\ne 0$,  we have  $J=\mathbb{R}=\psi(\p \Omega_{0,\infty,\theta_0}^T)=\psi(\p \Omega_{0,\infty,\theta_0}^B)$.

    \noindent Case 1. If $c=0$, and $u_r<0$ (resp. $u_r>0$) in $\overline{\Omega_{0,\infty,\theta_0}}\setminus\{\bm{0}\}$, then by similar arguments as in the proof of Theorem \ref{thm:4}, we have 
    \begin{equation*}
        \frac{\p^2 \Psi}{\p s^2}+\frac{\p^2 \Psi}{\p \theta^2}(\ln r_0,\theta)\leqslant 0, \quad (\text{resp. } \frac{\p^2 \Psi}{\p s^2}+\frac{\p^2 \Psi}{\p \theta^2}(\ln r_0,\theta) \geqslant 0), \quad \theta \in [0,\theta_0],
    \end{equation*}
    and 
    \begin{equation*}
        g(h(\theta))=\frac{\p^2 \Psi}{\p s^2}(\ln r_0,\theta)+\frac{\p^2 \Psi}{\p \theta^2} (\ln r_0,\theta), \quad \theta \in (0,\theta_0). 
    \end{equation*}
    Similar to the proof of Theorem \ref{thm:4}, it is easy to verify that $g$ is Lipschitz continuous in $J$, and that $g \leqslant 0$ (resp. $g \geqslant 0$). 
    Then, according to Lemma \ref{lem:4}, $\Psi$ is independent of $s$, that is, there exists a scalar function $h(\theta) \in C^3([0,\theta_0])$ such that $\Psi(s,\theta)=h(\theta)$. 
    Clearly, $h$ is strictly monotone with respect to $\theta$, since $|\p_\theta \Psi|>0$. From (\ref{thm6:eq3}), we know that $\psi(r,\theta)=h(\theta)$. Therefore, the results of Theorem \ref{thm:5} hold if (B1) or (B2) in Theorem \ref{thm:4} is satisfied.  
    
    \medskip

    \noindent Case 2. $c\ne 0$.  If $\p_\theta(ru_r)|_{\theta=0}=A$, then it follows from (\ref{thm0:eq0}) and (\ref{thm6:eq3}) that $\p^2_{\theta} \Psi(s,0)=-A$ for all $s \in \mathbb{R}$. Similarly, if $\p_\theta(ru_r)|_{\theta=\theta_0}=A$,  then $\p^2_{\theta} \Psi(s,\theta_0)=-A$ holds for all $s \in \mathbb{R}$. Following the same argument as in the proofs of Theorem \ref{thm:4}, we take the limit in (\ref{thm6:eq4}) as $\theta \to 0^+$ or $\theta \to \theta_0^-$, respectively. Together with (\ref{thm6:eq5}) and $J=\mathbb{R}=\psi(\p \Omega_{0,\infty,\theta_0}^T)=\psi(\p \Omega_{0,\infty,\theta_0}^B)$, it follows that there exists a constant $C$, such that $g(z)=C e^{-\frac{2}{c} z},\ z \in \mathbb{R}$. 
    Then, according to Lemma \ref{lem:4}, $\Psi$ is independent of $s$, that is, there exists a scalar function $h(\theta) \in C^3([0,\theta_0])$ such that $\Psi(s,\theta)=h(\theta)$. 
    Clearly, $h$ is strictly monotone with respect to $\theta$, since $|\p_\theta \Psi|>0$. From (\ref{thm6:eq3}), we know that $\psi(r,\theta)=c\ln r+h(\theta)$. 
    Thus, we have completed the proof.

    \bigskip
    
    \noindent \textbf{(ii)} $\alpha \neq 1$. 
    By (\ref{eq:BC:hom:-alpha:theta}) and (\ref{thm0:eq0}), we have  $\partial_r\psi(r, 0)=u_\theta|_{\theta=0}=c_1/r^\alpha$ and $\partial_r\psi(r, \theta_0)=u_\theta|_{\theta=\theta_0}=c_2/r^\alpha$. So there are some constants $A_1, A_2$, such that
    \begin{equation}\label{thm6:eqrmk}
        \psi(r,0)=\frac{c_1}{1-\alpha}r^{1-\alpha}+A_1, \quad \psi(r,\theta_0)=\frac{c_2}{1-\alpha}r^{1-\alpha} +A_2. 
    \end{equation}
    If $\alpha>1$, send $r\to \infty$ in the above, then by (\ref{thm6:eq1500}), we have $A_1=A_2$. If $\alpha<1$, send $r\to 0^+$ in the above, then by (\ref{thm6:eq1500}), we also have $A_1=A_2$.
    Without loss of generality, assume $A_1=A_2=0$, and
    \begin{equation}\label{thm6:eq7}
        \begin{split}
            \psi(r,0)=\frac{c_1}{1-\alpha}r^{1-\alpha}, \quad 
            \psi(r,\theta_0)=\frac{c_2}{1-\alpha}r^{1-\alpha} .
        \end{split}
    \end{equation}
    Similar as before, we have $|\p_\theta \psi|=|ru_r|>0$ in $\overline{\Omega_{0,\infty,\theta_0}} \setminus \{\bm{0} \}$.  So $c_1\ne c_2$.
    Define $J:=\{\psi(r,\theta): (r,\theta)\in \Omega_{0,\infty,\theta_0} \}$. By (\ref{thm6:eq7}), condition (i) in Lemma \ref{lem:2} holds. By Lemma \ref{lem:3}, there exists a function $g \in C^1(J)$ such that 
    \begin{equation}\label{thm6:eq8}
        \Delta \psi=\frac{\p^2 \psi}{\p r^2}+\frac{1}{r} \frac{\p \psi}{\p r}+\frac{1}{r^2} \frac{\p^2 \psi}{\p \theta^2}=g(\psi),  \text{ in } \Omega_{0,\infty,\theta_0}. 
    \end{equation}
    Let $s=\ln r$,  $D:=\mathbb{R} \times (0,\theta_0)$ and 
    \begin{equation}\label{thm6:eq9}
        \Psi(s,\theta):=\psi(e^s,\theta)e^{s(\alpha-1)}, \quad (s,\theta) \in \overline{D}. 
    \end{equation}
    Then $\Psi \in C^3(\overline{D})$ and $|\p_\theta \Psi|=|\partial_{\theta}\psi|e^{s(\alpha-1)}>0$. 
    By  (\ref{thm6:eq7})-(\ref{thm6:eq9}), we have
    \begin{equation}\label{thm6:eq10}
        \frac{\p^2 \Psi}{\p s^2}+ \frac{\p^2 \Psi}{\p \theta^2}+ 2(1-\alpha) \frac{\p \Psi}{\p s}+(1-\alpha)^2 \Psi =e^{s(1+\alpha)} g(e^{s(1-\alpha)} \Psi),\quad (s,\theta)\in D,
    \end{equation}
    with
    \begin{equation}\label{thm6:eq11}
        \Psi(s,0)=\frac{c_1}{1-\alpha},\quad \Psi(s,\theta_0)=\frac{c_2}{1-\alpha}. 
    \end{equation}
    If $\partial_\theta (r^{\alpha} u_r)|_{\theta=0} = c_3$, then by (\ref{thm0:eq0}) and (\ref{thm6:eq9}), we have
    \begin{equation}\label{thm6:eq12}
        \frac{\partial^2 \Psi}{\partial \theta^2}(s,0) = -c_3, \quad s\in \mathbb{R}.
    \end{equation}
    Similarly, if $\partial_\theta (r^{\alpha} u_r)|_{\theta=\theta_0} = c_4$, then
    \begin{equation}\label{thm6:eq13}
        \frac{\partial^2 \Psi}{\partial \theta^2}(s,\theta_0) = -c_4, \quad s\in \mathbb{R}.
    \end{equation}
    
    \noindent \textbf{Claim}: $\psi=h(\theta)r^{1-\alpha}$ in $\overline{\Omega_{0,\infty,\theta_0}}\setminus\{\bm{0}\}$ for some strictly monotone function $h(\theta)$.
        
    \noindent\emph{Proof of Claim}: We prove the claim under each of conditions (A1)-(A4) in Theorem \ref{thm:2} one by one. 
 
    \medskip

    \noindent (A1) In this case, $c_1 c_2>0$, $J=\psi(\p\Omega_{0,\infty,\theta_0}^T)=\psi(\p\Omega_{0,\infty,\theta_0}^B)$ is either $(0,+\infty)$ or $(-\infty,0)$.
    If $\p_\theta (r^{\alpha} u_r)|_{\theta=0}=c_3$, then (\ref{thm6:eq12}) holds. Taking limit $\theta \to 0^+$ in (\ref{thm6:eq10}), by (\ref{thm6:eq11}) and (\ref{thm6:eq12}), we conclude that there exists a constant $C$ such that 
    \begin{equation}\label{thm6:eq14}
        g(z)=C |z|^{\frac{\alpha+1}{\alpha-1}},\ z \in J. 
    \end{equation}
    If $\p_\theta (r^{\alpha} u_r)|_{\theta=\theta_0}=c_4$, send $\theta \to \theta_0^-$ in (\ref{thm6:eq10}). Then  (\ref{thm6:eq13}) implies that (\ref{thm6:eq14}) holds for some constant $C$.
    Substituted (\ref{thm6:eq14}) into (\ref{thm6:eq10}), we have
    \begin{equation*}
        \frac{\p^2 \Psi}{\p s^2}+ \frac{\p^2 \Psi}{\p \theta^2}+ 2(1-\alpha) \frac{\p \Psi}{\p s}+(1-\alpha)^2 \Psi =C |\Psi|^{\frac{\alpha+1}{\alpha-1}}.
    \end{equation*}
    Applying Lemma \ref{lem:4} with $L=\p_{s}^2+\p_{\theta}^2+2(1-\alpha)\p_s+(1-\alpha)^2$, $F=1$ and $g(\Psi)=C|\Psi|^{\frac{\alpha+1}{\alpha-1}}$ and $\Omega=D$ there.  Using the boundary conditions $(\ref{thm6:eq11})$ along with the fact $|\p_\theta \Psi|>0$, we have that $\Psi$ is independent of $s$. So there exists a strictly monotone function $h$ such that $\Psi(s,\theta)=h(\theta)$. 
    Therefore, by (\ref{thm6:eq9}), we know that $\psi(r,\theta)=h(\theta)r^{1-\alpha}$. 
    
    \medskip
    
    \noindent (A2) 
    In this case, we have $c_2 \neq c_1$, $c_1=0$,  $\alpha>1$ and (\ref{thm6:eq13}) holds. Therefore, $J=\psi(\p\Omega_{0,\infty,\theta_0}^T)$. By a similar argument as in (A1), it follows that $(\ref{thm6:eq14})$ holds for some constant $C$. Note that $g$ is Lipschitz in $J$, since $\alpha>1$. Lemma \ref{lem:4} implies that $\psi(r,\theta)=h(\theta)r^{1-\alpha}$. 
    
    \medskip
    
    \noindent (A3) 
    In this case, $c_1 \neq c_2$,  $c_2=0$, $\alpha>1$, $J=\psi(\p\Omega_{0,\infty,\theta_0}^B)$ and (\ref{thm6:eq12}) hold. By a similar argument as in (A2), it follows that $\psi(r,\theta)=h(\theta)r^{1-\alpha}$. 

    \medskip

    \noindent (A4) 
    In this case, $c_1 c_2<0$, $\alpha>1$, (\ref{thm6:eq12}) and (\ref{thm6:eq13}) both hold. Then $J=\mathbb{R}=\psi(\p\Omega_{0,\infty,\theta_0}^T) \cup \psi(\p\Omega_{0,\infty,\theta_0}^B) \cup \{0 \}$.  
    Taking the limits $\theta \to 0^+$ and $\theta \to \theta_0^-$ in (\ref{thm6:eq10}), together with (\ref{thm6:eq11})-(\ref{thm6:eq13}), it follows that there exist constants $C_1, C_2$ such that 
    \begin{equation*}
        g(z)=
        \begin{cases}
        C_1|z|^{\frac{\alpha+1}{\alpha-1}}, &z>0, \\
        C_2|z|^{\frac{\alpha+1}{\alpha-1}}, &z<0.
        \end{cases}
    \end{equation*}
    Since $\alpha>1$, it implies that $\displaystyle \lim_{z \to 0+} g(z)=\displaystyle \lim_{z \to 0-} g(z)=0$. Therefore, by the continuity of $g$, we can conclude that $g(0)=0$. From Lemma \ref{lem:4}, we know that $\psi(r,\theta)=h(\theta)r^{1-\alpha}$. The claim is proved.

    In all four cases above, we have obtained that $\psi(r,\theta)=h(\theta)r^{1-\alpha}$. The same argument as in the proof of Theorem \ref{thm:2} yields the assertion of Theorem \ref{thm:5}. 
\qed

\bigskip

Here we give a brief explanation of Remark \ref{rem:thm6}. 

From the proof of Theorem \ref{thm:5}, we see that condition (\ref{thm6:eq1500}) is used to derive (\ref{thm6:eq7}) from (\ref{thm6:eqrmk}). 
Under the alternative assumption (\ref{rem6:eq1}), we can still derive (\ref{thm6:eq7}) from (\ref{thm6:eqrmk}). 
We present the proof of this only for the case where $u_r<0$. 
Note $\p_\theta \psi=-ru_r>0$ in this case. Let $A_1, A_2$ be the constants in (\ref{thm6:eqrmk}). Taking the limit as $r \to \infty$ along $\theta=0$ and $\theta=\theta_0$ when $\alpha>1$, and as $r\to 0$ when $\alpha<1$, we conclude that $A_2 \geqslant A_1$. On the other hand, if (\ref{rem6:eq1}) holds, then
\begin{equation*}
    A_2-A_1+\frac{c_2-c_1}{1-\alpha}r_0^{1-\alpha}=\psi(r_0,\theta_0)-\psi(r_0,0)=-r_0\int_0^{\theta_0} u_r(r_0,\varphi) {\rm{d}}\varphi \leqslant \frac{c_2-c_1}{1-\alpha}r_0^{1-\alpha}. 
\end{equation*}
This implies $A_2 \leqslant A_1$. Therefore, $A_1=A_2$, and consequently, (\ref{thm6:eq7}) holds. 

\begin{proof}[Proof of Corollary \ref{cor:1}]
    From the result of Theorem \ref{thm:1}(i), Theorem \ref{thm:4}, and Theorem \ref{thm:5}(i), we know that $\vu$ is $(-1)$-homogeneous, $P$ is $(-2)$-homogeneous after subtracting a constant, and $(\vu,P)$ takes the following form 
    \begin{equation*}
        \vu=\frac{c}{r} \bm{e_\theta}+\frac{f(\theta)}{r} \bm{e_r}, \quad P=\frac{p}{2r},
    \end{equation*}
    where $p$ is a constant, and $f\in C^2([0,\theta_0])$ is strictly sign-definite and satisfies 
    \begin{equation}\label{eq:cor1}
        cf'=f^2+c^2+2p.
    \end{equation}
    When $\theta_0=2\pi$, it naturally follows that $f(0)=f(2\pi)$. 

    We claim that $f$ must be a nonzero constant (denoted by $d$), and therefore $p=-\frac{c^2+d^2}{2}$ when $\theta_0=2\pi$. 
    
    Indeed,  if $c=0$, this claim follows straightforwardly from (\ref{eq:cor1}) and the assumption $|u_r|>0$.   
    
     If $c \neq 0$, we first integrate both sides of equation (\ref{eq:cor1}) with respect to $\theta$ over the interval $[0,2\pi]$, and then use the periodicity condition $f(0) = f(2\pi)$ along with the strict sign-definiteness of $f$ to obtain $c^2+2p<0$. 
    
    Let $w(\theta):=\exp (-\int_0^\theta \frac{f(s)}{c} {\rm{d}}s)$, then $w\in C^3(0,2\pi)$ is non-negative and $f=-cw'/w$. Direct calculation shows that $w$ satisfies the following equation
    \begin{equation*}
        w''=\lambda w, 
    \end{equation*}
    where $\lambda=-\frac{c^2+2p}{c^2}>0$. This implies that $w(\theta)=C_1 e^{\sqrt{\lambda}\theta}+C_2 e^{-\sqrt{\lambda}\theta}$ for some constants $C_1$, $C_2$. Therefore, 
    \begin{equation*}
        f(\theta)=-c\frac{w'}{w}=-c\frac{C_1 \sqrt{\lambda} e^{\sqrt{\lambda}\theta}-C_2 \sqrt{\lambda} e^{-\sqrt{\lambda}\theta}}{C_1 e^{\sqrt{\lambda}\theta}+C_2 e^{-\sqrt{\lambda}\theta}}.
    \end{equation*}
    From $f(0)=f(2\pi)$ and $c\neq 0$, we know that $C_1C_2=0$. Therefore, $f$ is a nonzero constant. Hence, Corollary \ref{cor:1} holds. 
\end{proof}

\section*{Appendix: Explicit expressions of homogeneous solutions}
We provide explicit expressions for the $(-\alpha)$-homogeneous solutions of (\ref{eq:euler}) in the region (\ref{eq:domain}). 

In polar coordinates $(r,\theta)$, the $(-\alpha)$-homogeneous solutions $(\vu,P)$ can be written as 
\begin{equation*}
    \vu=\frac{v(\theta)}{r^\alpha} \bm{e_\theta}+\frac{f(\theta)}{r^\alpha}\bm{e_r}, \quad P=\frac{p(\theta)}{r^{2\alpha}}, 
\end{equation*}
where $\bm{e_\theta}, \bm{e_r}$ are two unit vectors 
\begin{equation*}
    \bm{e_\theta}=(-\sin \theta,\cos \theta)^T, \qquad  \bm{e_r}=(\cos \theta,\sin \theta)^T.
\end{equation*}

A direct calculation shows that $(-\alpha)$-homogeneous solutions satisfy the following equations
\begin{equation}\label{App:eq1}
    \begin{cases}
    \begin{aligned}
    vf'-\alpha f^2-v^2-2\alpha p &=0, \\
    (1-\alpha)fv+vv'+p'&=0, \\
    (1-\alpha)f+v' &=0.
    \end{aligned}
    \end{cases}
\end{equation}
Here, $'$ represents the derivative with respect to the variable $\theta$. By the second line and the third line of (\ref{App:eq1}), it follows that $p$ must be constant for any $\alpha \in \mathbb{R}$, and (\ref{App:eq1}) reduces to 
\begin{equation}\label{App:eq2}
    \begin{cases}
    \begin{aligned}
    vf'-\alpha f^2-v^2-2\alpha p &=0, \\
    (1-\alpha)f+v' &=0.
    \end{aligned}
    \end{cases}
\end{equation}

\subsection*{A. $(-1)$-homogeneous solutions}
We first consider the case $\alpha=1$. The second line of (\ref{App:eq2}) implies that $v$ must be constant, and we obtain a single equation for $f$,
\begin{equation}\label{App:eq3}
    vf'=f^2+v^2+2p.
\end{equation}

When $v=0$, equation (\ref{App:eq3}) means that $f^2=-2p$, thus $p\leqslant 0$ and 
\begin{equation*}
    \vu=\pm \frac{\sqrt{-2p}}{r}\bm{e_r}, \quad P=\frac{p}{r^2}.
\end{equation*}

When $v \neq 0$ and $v^2+2p>0$, (\ref{App:eq3}) has a class of solutions 
\begin{equation}\label{sol:1}
        f(\theta) =\sqrt{v^2+2p} \tan \bigg(\frac{\sqrt{v^2+2p}}{v} \theta+C \bigg),
\end{equation}
where $C$ is an arbitrary constant that makes $f(\theta)$ defined in $[0,\theta_0]$. Hence, we have 
\begin{equation*}
    \vu=\frac{v}{r} \bm{e_\theta}+\frac{
    f(\theta)}{r}\bm{e_r}, \quad P=\frac{p}{r^2},
\end{equation*}
where $f(\theta)$ is given by (\ref{sol:1}).

When $v \neq 0$ and $v^2+2p=0$, all solutions of (\ref{App:eq3}) are given by 
\begin{equation}\label{sol:2}
    f(\theta)\equiv 0, \quad  {\rm{or}} \quad  f(\theta)=-\frac{v}{\theta+C},
\end{equation}
where $C$ is an arbitrary constant such that $f(\theta)$ is defined in $[0,\theta_0]$. Then we have 
\begin{equation*}
    \vu=\frac{v}{r} \bm{e_\theta}+\frac{
    f(\theta)}{r}\bm{e_r}, \quad P=\frac{-v^2}{2 r^2}, 
\end{equation*}
where $f(\theta)$ is given by (\ref{sol:2}). 

When $v \neq 0$ and $v^2+2p<0$, (\ref{App:eq3}) has a class of solutions 
\begin{equation}\label{sol:3}
    \begin{split}
        f(\theta)&\equiv \pm \sqrt{-(v^2+2p)}, \quad {\rm{or}} \\ 
        f(\theta)&=\frac{2 \sqrt{-(v^2+2p)}}{1+C \exp{\Big( \frac{2 \sqrt{-(v^2+2p)}}{v} \theta \Big) } } -\sqrt{-(v^2+2p)}.
    \end{split}
\end{equation}
Here, $C$ is an arbitrary non-zero constant that ensures $f(\theta)$ is defined in $[0,\theta_0]$. Therefore, 
\begin{equation*}
    \vu=\frac{v}{r} \bm{e_\theta}+\frac{
    f(\theta)}{r}\bm{e_r}, \quad P=\frac{p}{r^2}, 
\end{equation*}
where $f(\theta)$ satisfies (\ref{sol:3}). 

\subsection*{B. $(-\alpha)$-homogeneous solutions with $\alpha \neq 1$}
We next consider the case that $\alpha \neq 1$. Explicit formulas for $v(\theta)$ and $f(\theta)$ in this case are not always available.  
Now, we give some explicit expressions for $v, f$ when $p \leqslant 0$. 

From the second line in (\ref{App:eq2}), we know that 
\begin{equation}\label{App:eq4}
    f(\theta)=\frac{v'(\theta)}{\alpha-1}.
\end{equation}
Substituting (\ref{App:eq4}) into the first line of (\ref{App:eq2}), we find that
\begin{equation}\label{App:eq5}
     \frac{-\alpha}{(1-\alpha)^2} (v')^2-\frac{1}{1-\alpha}v v''-v^2=2\alpha p.
\end{equation}

When $p=0$, it is easy to check that (\ref{App:eq5}) has  solutions 
\begin{equation}\label{App:eq10}
    v(\theta)=C_1 \cos^{1-\alpha}(\theta+C_2).
\end{equation}
Then
\begin{equation}\label{App:eq11}
    f(\theta)=C_1 \cos^{-\alpha}(\theta+C_2) \sin(\theta+C_2),
\end{equation}
where $C_1, C_2$ are arbitrary constants such that $v(\theta)$ and $f(\theta)$ are defined in $[0,\theta_0]$. Thus, 
\begin{equation*}
     \vu=\frac{v(\theta)}{r^\alpha} \bm{e_\theta}+\frac{f(\theta)}{r^\alpha}\bm{e_r}, \quad P=0,
\end{equation*}
where $v(\theta)$ is given by (\ref{App:eq10}), and $f(\theta)$ is given by (\ref{App:eq11}).

When $p<0$, it is easy to verify that (\ref{App:eq5}) has  a class of non-trivial solutions 
\begin{equation}\label{App:eq20}
    v(\theta)=\sqrt{-2p} \sin[(1-\alpha)\theta+C].
\end{equation}
Then 
\begin{equation}\label{App:eq21}
    f(\theta)=-\sqrt{-2p} \cos[(1-\alpha)\theta+C],
\end{equation}
where $C$ is an arbitrary constant. Thus, 
\begin{equation*}
     \vu=\frac{v(\theta)}{r^\alpha} \bm{e_\theta}+\frac{f(\theta)}{r^\alpha}\bm{e_r}, \quad P=\frac{p}{r^{2\alpha}},
\end{equation*}
where $v(\theta)$ is given by (\ref{App:eq20}), and $f(\theta)$ is given by (\ref{App:eq21}).

\bibliographystyle{plain}
\bibliography{Reference.bib}

\end{document}